\documentclass[letterpaper, two column, 10 pt, conference]{ieeeconf}  

\IEEEoverridecommandlockouts                              
\usepackage[utf8]{inputenc}
\usepackage{amsmath,bm}
\usepackage{amsfonts}
\usepackage{amssymb}
\usepackage{graphicx}
\usepackage{amsfonts}
\usepackage{hyperref}
\usepackage{caption}
\usepackage{subcaption}
\usepackage{algorithm}
\usepackage{algpseudocode}
\usepackage{cases}
\usepackage{enumerate}
\usepackage{xcolor}
\usepackage{cite}
\usepackage{verbatim}
\usepackage[normalem]{ulem}

\newtheorem{definition}{Definition}
\newtheorem{Theorem}{Theorem}

\newtheorem{Lemma}{Lemma}

\newtheorem{Remark}{Remark}

\newcommand{\m}[1]{\mathbf{#1}}
\newcommand{\mc}[1]{\mathcal{#1}}
\newcommand{\mb}[1]{\mathbb{#1}}

\newcommand\numberthis{\addtocounter{equation}{1}\tag{\theequation}}
\newcommand*{\eprf}{\hfill\ensuremath{\blacksquare}}

\usepackage{tikz}
	\usetikzlibrary{shapes,arrows}
	\tikzstyle{frame} = [draw, -latex]
	\tikzstyle{line} = [draw]
	\tikzstyle{line2} = [draw, dashdotted]
	\tikzstyle{line3} = [draw, dashed]
	\tikzstyle{line3UD} = [draw, dashed]
	\tikzstyle{place} = [circle, draw=black, fill=white, thick, inner sep=2pt, minimum size=1mm]
	\tikzstyle{place2} = [circle, draw=black, fill=black, thick, inner sep=2pt, minimum size=1mm]
	\tikzstyle{placeRed} = [circle, draw=red, fill=red, thick, inner sep=2pt, minimum size=1mm]
	\tikzstyle{vertex} = [circle, draw=black, fill=black, thick, inner sep=2pt, minimum size=1mm]
\usepackage{tkz-euclide}
\tikzset{
    right angle quadrant/.code={
        \pgfmathsetmacro\quadranta{{1,1,-1,-1}[#1-1]}     
        \pgfmathsetmacro\quadrantb{{1,-1,-1,1}[#1-1]}},
    right angle quadrant=1, 
    right angle length/.code={\def\rightanglelength{#1}},   
    right angle length=1.4ex, 
    right angle symbol/.style n args={3}{
        insert path={
            let \p0 = ($(#1)!(#3)!(#2)$) in     
                let \p1 = ($(\p0)!\quadranta*\rightanglelength!(#3)$), 
                \p2 = ($(\p0)!\quadrantb*\rightanglelength!(#2)$) in 
                let \p3 = ($(\p1)+(\p2)-(\p0)$) in  
            (\p1) -- (\p3) -- (\p2)
        }
    }
}

\title{\LARGE \bf
Distributed Computation of Graph Matching in Multi-Agent Networks
}

\author{   Quoc Van Tran, Zhiyong Sun, Brian D. O. Anderson, and Hyo-Sung Ahn
\thanks{Q. V. Tran and H.-S. Ahn are with the School of Mechanical Engineering,
        Gwangju Institute of Science and Technology, Gwangju, Republic of Korea.
        E-mails: {\tt\small $\{$tranvanquoc, hyosung$\}$@gist.ac.kr}}
\thanks{Z. Sun is with Department of Electrical Engineering, Eindhoven University of Technology (TU/e), the Netherlands. E-mails: {\tt\small sun.zhiyong.cn@gmail.com, z.sun@tue.nl}}
\thanks{B. D. O. Anderson is with the Research School of Electrical, Energy
and Materials Engineering, Australian National University, Acton, A.C.T.
2601, Australia, the School of Automation, Hangzhou Dianzi University,
Hangzhou 310018, China, and the Data61-CSIRO, Canberra, A.C.T. 2601,
Australia (\tt\small brian.anderson@anu.edu.au).}
}

\begin{document}

\maketitle
\thispagestyle{empty}
\pagestyle{empty}

\begin{abstract}
This work considers the distributed computation of the one-to-one vertex correspondences between two undirected and connected graphs, which is called \textit{graph matching}, over multi-agent networks. Given two \textit{isomorphic} and \textit{asymmetric} graphs, there is a unique permutation matrix that maps the vertices in one graph to the vertices in the other. 
Based on a convex relaxation of graph matching in Aflalo et al. \cite{Aflalo2015pnas}, we propose a distributed computation of graph matching as a distributed convex optimization problem subject to equality constraints and a global set constraint, using a network of multiple agents whose interaction graph is connected. Each agent in the network only knows one column of each of the adjacency matrices of the two graphs, and all agents collaboratively learn the graph matching by exchanging information with their neighbors. The proposed algorithm employs a projected primal-dual gradient method to handle equality constraints and a set constraint. Under the proposed algorithm, the agents' estimates of the permutation matrix converge to the optimal permutation globally and exponentially fast. Finally, simulation results are given to illustrate the effectiveness of the method.
\end{abstract}

\section{Introduction} 
A graph consists of a set of vertices and a set of edges. The vertices might represent abstract entities such as features in an image, point patterns or users in social networks, or represent physical agents such as body parts, mobile robots or unmanned aerial vehicles, and the edges in the graph represent the relations between the vertices. Given two connected graphs, the problem of finding an optimal permutation matrix that minimizes the disagreement between two corresponding adjacency matrices $\m{A}$ and $\m{B}\in \mb{R}^{n\times n}$ is referred to as \textit{graph matching}. Graph matching has a wide range of applications in different science and engineering disciplines such as computer vision and pattern recognition \cite{Torresani2008eccv,Foggia2014,PDas2012}, neuroscience \cite{Vogelstein2015PloSONE}, and formation control \cite{ZKan2015tcns, Sakurama2019tac}, to name a few.

Graph matching has been studied extensively in the last few decades. 
Though there are numerous graph matching algorithms developed in the literature, graph matching remains computationally intractable \cite{Lyzinski2016pami}. 
Graph matching with zero adjacency disagreement is said to be \textit{exact matching}. In the presence of noise, it is referred to as \textit{inexact matching} with the minimal adjacency disagreement. 
Heuristic algorithms, e.g., ones based on some forms of tree search \cite{Cook2003, Ullmann2011}, have no theoretical guarantee of the convergence to the global minimizer of the graph matching. Spectral methods rely on the fact that two adjacency matrices of two isomorphic graphs share the same spectrum \cite{Caelli2004pami, Duchenne2011pami}.
A popular approach for addressing graph matching is the continuous optimization based on relaxations of the discrete graph matching problem \cite{Aflalo2015pnas, Lyzinski2016pami}. Two common relaxations of graph matching are the \textit{indefinite relaxation} by minimizing $-\langle \m{AP},\m{PB}\rangle$ \cite{Vogelstein2015PloSONE, Lyzinski2016pami}, where $\langle \cdot,\cdot\rangle$ denotes the inner product, and the \textit{convex relaxation} by minimizing $||\m{AP}-\m{PB}||_F^2$ \cite{Aflalo2015pnas, Lyzinski2016pami}, where $||\cdot||_F$ denotes the Frobenius norm, over $\m{P}$ in the set of doubly stochastic matrices. The former in general has multiple local minima and hence depending on the initialization, a gradient-based algorithm will converge to a local minimum of the objective function, but not necessarily the global minimum. The convex relaxation has a unique least-squares solution in the doubly stochastic matrix set. However, the actual permutation that matches two isomorphic graphs can be only recovered when the norm of the perturbed adjacency matrix is less than a small value \cite{Lyzinski2016pami}. By using friendliness properties of adjacency matrices, which are characterized by their spectral properties, the graph matching can be further relaxed by replacing the set of doubly stochastic matrices by the set of \textit{pseudo-stochastic} matrices $\{\m{P}:\m{P}\m{1}_n=\m{1}_n\}$ \cite{Aflalo2015pnas}.
We refer the readers to \cite{Foggia2014,Lyzinski2016pami} for more comprehensive reviews of graph matching algorithms.
 
This paper considers the distributed computation of graph matching over an $n$-agent network in which each agent in the network only knows one column of each of the adjacency matrices $\m{A}$ and $\m{B}$. The agents in the network collaboratively learn the graph matching by exchanging information with their neighbors. The distributed setup is commonly used in distributed algorithms to solve linear algebraic equations \cite{Mou2015tac,Anderson2016naco, PWang2019survey} and linear matrix equations \cite{Deng2019, XZeng2019tac}. Distributed algorithms and distributed optimizations over networked systems in particular have attracted lot of research interest in different areas of science and engineering \cite{He2018NIPS, KZhang2018icml, TYang2019, KRyu2019cdc, Alghunaim2019NIPS}, partly due to the increasing scale of the underlying problems, the distributed nature of networked systems, and the privacy of individual information. Motivated by these facts, we present a distributed algorithm to compute the graph matching between two isomorphic graphs based on a convex relaxation of graph matching, and by using distributed optimization over multi-agent networks.

The contributions of this paper are as follows. First, based on the centralized algorithm in \cite{Aflalo2015pnas}, we formulate the graph matching as a multi-agent convex optimization problem subject to equality constraints and a global set constraint, in which each agent in the network only knows one column of each of the adjacency matrices of the two graphs to be matched. Then as a development of \cite{Aflalo2015pnas}, we prove that almost all adjacency matrices have friendliness properties, which allows us to focus on graph matching of asymmetric graphs without loss of generality.
As the second contribution, we describe the geometric interpretation of the constrained set and derive an orthogonal projection operator associated with it. We then propose a distributed optimization algorithm to compute the permutation matrix that matches two isomorphic and asymmetric graphs, over the multi-agent network. Further, we establish the globally exponential convergence of the agents' estimates of the permutation matrix to the actual one, assuming that the interaction graph of the agents is connected. Finally, we illustrate the theoretical analysis through simulations.

The rest of the paper is organized as follows. Preliminaries and the graph matching problem are presented in Section \ref{sec:preliminary}. In Section \ref{sec:distributed_algorithm}, we investigate a projected optimization algorithm over a multi-agent system to compute the graph matching. An example of matching two isomorphic and asymmetric graphs is presented in Section \ref{sec:sim}. Finally, Section \ref{sec:conclusion} concludes this paper.

\section{Preliminaries and Problem Formulation}\label{sec:preliminary} 
\subsubsection*{Notation}
Let $\mc{G}_1=(\mc{V},\mc{E}_1,\m{A})$ and $\mc{G}_2=(\mc{V},\mc{E}_2,\m{B})$ be two undirected graphs of $n$ vertices whose index set is $\mc{V}=\{1,\ldots,n\}$ and edge sets are $\mc{E}_1,\mc{E}_2\subseteq \mc{V}\times \mc{V}$, respectively. In addition, $\m{A}$ and $\m{B}\in [0,\infty)^{n\times n}$ denote the symmetric adjacency matrices\footnote{In this work, graphs are considered to be weighted without further explicit mention.} of the undirected graphs $\mc{G}_1$ and $\mc{G}_2$, respectively,
whose entries are non-negative scalar weights characterizing the interactions between the vertices. 
Let $\m{a}_i$ and $\m{b}_i\in \mb{R}^{n}$ be the $i$th column vectors of $\m{A}$ and $\m{B}$, respectively.
The space of vertex permutations is denoted by $\mc{P}=\{\pi:\mc{V}\rightarrow\mc{V}\}$, which is characterized by the set of permutation matrices $\{\m{\Pi}\in \{0,1\}^{n\times n}:\m{\Pi}\m{1}_n=\m{\Pi}^\top\m{1}_n=\m{1}_n\}$, where $\m{1}_n$ is the vector of all ones. Let $\text{col}(\m{x}_1,\ldots,\m{x}_n)=[\m{x}_1^\top,\ldots,\m{x}_n^\top]^\top$ be the stack vector of $\m{x}_1,\ldots,\m{x}_n\in\mb{R}^n$.
Denote the inner product of two vectors or two matrices of the same size as $\langle\Phi,\m{X} \rangle=\sum_{i,j}(\Phi)_{ij}(\m{X})_{ij}$, where $(\cdot)_{ij}$ is the $(i,j)$-th entry. Let $(\m{X})_{i}^\text{C}$ and $(\m{X})_i^\text{R}$ be the $i$th column vector and the $i$th row vector of a matrix $\m{X}$, respectively. The notation $||\cdot||$ denotes the Euclidean norm. Let $\m{R}_+^n$ be the nonnegative orthant of $\mb{R}^n$. Denote by $\mc{B}_\epsilon(\m{x}),\m{x}\in \mb{R}^n$ as the open ball \textit{centered} at a point $\m{x}$ with \textit{radius} $\epsilon>0$.
\subsection{Convex Analysis}
A set $\Omega\subseteq \mb{R}^n$ is \textit{convex} if for any $\m{x}$ and $\m{y}\in\Omega$ and $\alpha\in [0,1]$, it holds $\alpha \m{x}+(1-\alpha)\m{y}\in \Omega$. A function $f:\Omega\rightarrow \mb{R}$ is \textit{convex} if 
$
f(\alpha\m{x}+(1-\alpha)\m{y})\leq \alpha f(\m{x})+(1-\alpha)f(\m{y}),
$
for any $\m{x}$ and $\m{y}\in\Omega$ and $\alpha\in [0,1]$.
\subsection{Convergence Analysis}
Consider the autonomous system
\begin{equation}\label{eq:auto_system}
\dot{\m{x}}=f(\m{x}),~\m{x}(0)=\m{x}_0,
\end{equation}
where $f:\Omega\rightarrow\mb{R}^n$ is a Lipschitz continuous map from a set $\Omega\subset\mb{R}^n$ into $\mb{R}^n$. Let $\bar{\m{x}}\in \Omega$ be an equilibrium point of \eqref{eq:auto_system}. Consider a solution trajectory $\m{x}(t)$ of \eqref{eq:auto_system}. A point $\m{p}$ is said to be a \textit{positive limit point} of $\m{x}(t)$ if there exists a sequence $\{t_m\}$, with $t_m\rightarrow \infty$ as $m\rightarrow\infty$, such that $\m{x}(t_m)\rightarrow \m{p}$ as $m\rightarrow \infty$. A set $\mc{M}$ is said to be a \textit{positively invariant set} if $\m{x}(0)\in \mc{M}\Rightarrow\m{x}(t)\in \mc{M},\forall t\geq 0$.
The equilibrium point $\bar{\m{x}}$ of \eqref{eq:auto_system} is \textit{stable} if, for each $\epsilon>0$, there is $\delta=\delta(\epsilon)>0$ such that $\m{x}(0)\in \mc{B}_\delta(\bar{\m{x}})\Rightarrow \m{x}(t)\in\mc{B}_\epsilon(\bar{\m{x}}),\forall t\geq 0$. 
\begin{Lemma}\cite[Theorem 4.4]{Khalil2002}\label{lm:LaSalle_principle}
Let $\Omega\in \mb{R}^n$ be a compact set that is positively invariant with respect to \eqref{eq:auto_system}. Let $V:\Omega\rightarrow\mb{R}$ be a continuous differentiable positive definite function such that $\dot{V}(\m{x})\leq 0$ in $\Omega$. Let $S=\{\m{x}\in \Omega:\dot{V}(\m{x})=0\}$. Let $M$ be the largest invariant set in $S$. Then, every solution starting in $\Omega$ approaches $M$ as $t\rightarrow \infty$.
\end{Lemma}
\begin{Lemma}\cite[Lemma 2.2]{XZeng2019tac}\label{lm:exponen__converg_linear_system}
Suppose that the system \eqref{eq:auto_system} is a linear time-invariant system, i.e. $f(\m{x})=\m{M}\m{x}+\m{b}$, where $\m{M}\in \mb{R}^{n\times n}$, $\m{b}\in \mb{R}^n$, and $\Omega=\mb{R}^n$. Then, if the system \eqref{eq:auto_system} converges to an equilibrium for any initial condition, its convergence is exponentially fast.
\end{Lemma}
\subsection{Graph Matching}
A permutation $\pi$ maps a vertex $i$ in $\mc{G}_1$ to a vertex $\pi_i$ in $\mc{G}_2$, and associates each $(i,j)$-th entry of $\m{A}$ to an entry $(\m{B})_{\pi_i\pi_j}$ in $\mc{G}_2$. Let $\m{\Pi}$ be the permutation matrix associated with this permutation $\pi$. Then, $\m{A}=\m{\Pi}^\top\m{B}\m{\Pi}$ for exact matching. We denote by 
\begin{equation}\label{eq:distortion_function}
\text{dis}_{\mc{G}_1\rightarrow\mc{G}_2}(\m{\Pi})=||\m{A}-\m{\Pi}^\top\m{B}\m{\Pi}||_F
\end{equation}
the \textit{distortion} function specifying the adjacency disagreement between $\m{A}$ and $\m{B}$. Two graphs $\mc{G}_1$ and $\mc{G}_2$ are said to be \textit{isomorphic} if their adjacency disagreement is zero in the sense of Eq. \eqref{eq:distortion_function}, for some permutations $\m{\Pi}$.
Denote by $\text{Iso}(\mc{G}_1\rightarrow\mc{G}_2)=\{\m{\Pi}\in \mc{P}:\text{dis}_{\mc{G}_1\rightarrow\mc{G}_2}(\m{\Pi})=0\}$ the set of all permutations, or i.e., \textit{isomorphisms}, matching $\mc{G}_1$ and $\mc{G}_2$.

A graph $\mc{G}=(\mc{V},\m{A})$ is \textit{symmetric}\footnote{The symmetry or asymmetry of a graph should be distinguished from the symmetry of its associated adjacency matrix. The latter itself is \textit{symmetric} simply when the graph $\mc{G}$ is undirected and connected.} if there exists a nontrivial permutation $\m{\Pi}\in \mc{P}$ such that $\text{dis}_{\mc{G}\rightarrow\mc{G}}(\m{\Pi})=0$, or i.e., $\mc{G}$ has a nontrivial \textit{automorphism group} $\text{Iso}(\mc{G}\rightarrow\mc{G})$. The graph $\mc{G}$ is \textit{asymmetric} if it is not symmetric, or equivalently the only permutation matrix satisfying $\text{dis}_{\mc{G}\rightarrow\mc{G}}(\m{\Pi})=0$ is the trivial identity matrix.
If two asymmetric graphs $\mc{G}_1$ and $\mc{G}_2$ are isomorphic, they are related by a unique permutation, denoted by $\m{\Pi}^*$, which is the global solution of problem \eqref{eq:graph_matching} below. On the other hand, given two isomorphic and symmetric graphs $\mc{G}_1$ and $\mc{G}_2$, there are two or more permutation matrices which satisfy $\text{dis}_{\mc{G}_1\rightarrow\mc{G}_2}=0$.
\begin{Remark}
The symmetry/asymmetry property defined above for a graph, which is characterized by the uniqueness of its automorphisms, is also different from \textit{structural symmetry/structural asymmetry} of the associated graph topology.
In Fig. \ref{fig:example_graphs}, $\mc{G}_1$ and $\mc{G}_3$ are \textit{structurally symmetric} with respect to their vertical axes. That is, the permutations of the corresponding vertices on the two sides of the vertical dashed line leave the graphs unchanged. As a comparison, the only permutation under which the graphs $\mc{G}_2$ and $\mc{G}_4$ are invariant is the identity mapping. As a result, $\mc{G}_2$ and $\mc{G}_4$ in Fig. \ref{fig:example_graphs} are \textit{structurally asymmetric}. A structurally asymmetric graph is also asymmetric for all the edge weights in its associated adjacency matrices. A symmetric graph has a structurally symmetric graph topology and has some pairs of edges in which edges in each pair have the same weight. For instance, $\mc{G}_3$ in Fig. \ref{fig:example_graphs} is symmetric when the weights of every two corresponding edges on the two sides of the vertical dashed line are equal.
\end{Remark}
\begin{figure}[t]
\centering
\begin{tikzpicture}[scale=1.2]
\node[place] (4) at (0.6,0.) [] {};
\node[place] (3) at (0,0) [] {};
\node[place] (1) at (0,1.2) [] {};
\node[place] (2) at (0,0.6) [] {};
\node[place] (5) at (0.6,0.6) [] {};
\node[place] (6) at (0.6,1.2) [] {};
\node[] (g3) at (0.5,-0.2) [label=below:$\mathcal{G}_4$] {};

\draw (2) -- node [left] {} (3);
\draw (3)  -- node [left] {} (4);
\draw (4)  -- node [left] {} (5);
\draw (2)  -- node [left] {} (5);
\draw (2)  -- node [left] {} (1);
\draw (1) -- node [left] {} (5);
\draw (6) -- node [left] {} (5);

\node[place] (a4) at (-2.5-0.5,0) [] {};
\node[place] (a5) at (-4.5-0.5,0) [] {};
\node[place] (a3) at (-3-0.5,0) [] {};
\node[place] (a1) at (-4-0.5,0) [] {};
\node[place] (a2) at (-3.5-0.5,0) [] {};
\node[place] (a6) at (-3.75-0.5,.5) [] {};
\node[] (g1) at (-3.5-0.5,-0.2) [label=below:$\mathcal{G}_2$] {};
\draw (a1) -- (a5) (a1) -- (a2) (a2)--(a3) (a3)--(a4) (a6)--(a1) (a6)--(a2);

\node[place] (a4) at (-2.5-0.5,1.3) [] {};
\node[place] (a5) at (-4.5-0.5,1.3) [] {};
\node[place] (a3) at (-3-0.5,1.3) [] {};
\node[place] (a1) at (-4-0.5,1.3) [] {};
\node[place] (a2) at (-3.5-0.5,1.3) [] {};
\node[] (g1) at (-3.-0.5,1.1) [label=below:$\mathcal{G}_1$] {};
\draw[red, dashed] (-3.5-0.5,1.1) -- (-3.5-0.5,1.55);
\draw (a1) -- (a5) (a1) -- (a2) (a2)--(a3) (a3)--(a4);

\node[place] (b4) at (-1.6,0.) [] {};
\node[place] (b3) at (-1,0) [] {};
\node[place] (b1) at (-1,1.2) [] {};
\node[place] (b2) at (-1,0.6) [] {};
\node[place] (b5) at (-1.6,0.6) [] {};
\node[place] (b6) at (-1.6,1.2) [] {};
\node[] (g3) at (-1.4,-0.2) [label=below:$\mathcal{G}_3$] {};
\draw[red, dashed] (-1.3,1.4) -- (-1.3,-0.15);
\draw (b1) -- (b2) (b2)--(b3) (b3)--(b4) (b5)--(b6) (b5)--(b2) (b5)--(b4);
\end{tikzpicture}
\caption{Examples of structurally symmetric and structurally asymmetric graphs. The graphs $\mc{G}_1$ and $\mc{G}_3$ are structurally symmetric w.r.t. their vertical axes (dashed lines); $\mc{G}_2$ and $\mc{G}_4$ are structurally asymmetric.}
\label{fig:example_graphs}
\end{figure}
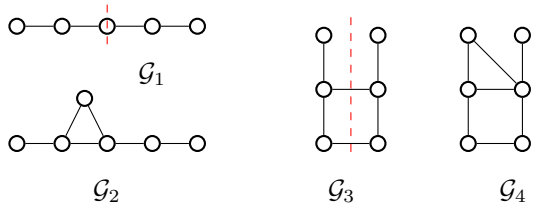

We now define a subclass of graphs which both covers almost all graphs with a given number of nodes, and is the subclass which we will prove to be as large as the set of asymmetric graphs.
\begin{definition}[Friendly Graphs]\cite{Aflalo2015pnas}\label{def:friendly_graphs}
A graph $\mc{G}$ is said to be \textit{friendly} if its adjacency matrix $\m{A}$ has simple spectrum (all eigenvalues are distinct) and eigenvectors satisfy $\m{u}_i^\top\m{1}_n\neq 0$, for all $i=1,\ldots,n$.
\end{definition}
The properties of adjacency matrices of friendly graphs are important in convexly relaxing the graph matching problem in the next subsection. A friendly graph is necessarily asymmetric \cite[Lemma 1]{Aflalo2015pnas}. In addition, it will be shown below that almost all adjacency matrices of asymmetric graphs are also friendly. That is, the properties of the adjacency matrices of unfriendly graphs are nongeneric. In fact, it is shown in \cite{Tao2017} that almost all adjacency matrices of random graphs have simple spectrum. 
In addition, in the sequel, we show that almost all adjacency matrices with nonnegative entries have eigenvectors that are not orthogonal to a given nonzero vector, e.g., $\m{1}_n$.

\begin{Theorem}\label{thm:zero_measure_set}
The set of weighted adjacency matrices which have nonnegative entries and an eigenvector orthogonal to $\m{1}_n$ is a set of measure zero.
\end{Theorem}
\begin{proof}
See Appendix \ref{app:zero_measure_set}.
\end{proof}
We, therefore, focus on matching two asymmetric graphs by implicitly assuming friendliness for each, which holds almost surely according to above results.
%

\subsection{Convex relaxation of graph matching}
The graph matching problem is to find an optimal permutation matrix $\m{\Pi}^*$ satisfying the following optimization problem:
\begin{align*}\label{eq:graph_matching}
(\text{GM})\qquad\m{\Pi}^*&=\underset{\m{\Pi}\in\mc{P}}{\text{argmin}}||\m{A}-\m{\Pi}^\top\m{B}\m{\Pi}||_F^2\\
&=\underset{\m{\Pi}\in\mc{P}}{\text{argmin}}||\m{\Pi}\m{A}-\m{B}\m{\Pi}||_F^2.\numberthis
\end{align*}
where the last equality follows from the unitarity property of permutation matrices. It is noted that although the objective function in the problem \eqref{eq:graph_matching} is a convex function of the argument $\m{\Pi}$, the permutation set $\mc{P}$ is non-convex and so the problem \eqref{eq:graph_matching} is not a convex problem. Although the set is finite, it has $n!$ elements, which means that for large $n$, exhaustive search will become infeasible. Thus, it is desirable to replace $\mc{P}$ with a larger convex set. For example, consider the space of \textit{pseudo-stochastic} matrices $\mc{D}_n:=\{\m{P}:\m{P}\m{1}_n=\m{1}_n\}$ and the following relaxation of the graph matching problem (RGM):
\begin{equation}\label{eq:RGM}
(\text{RGM})~ \m{P}^*=\underset{\m{P}\in \mc{D}_n}{\text{argmin}}||\m{P}\m{A}-\m{B}\m{P}||_F^2,.
\end{equation}
We focus in this paper just on the issue of finding the solution to the \textit{isomorphic graph matching} problem, i.e., when the two graphs in question are isomorphic. We now have the following lemma.
\begin{Lemma}\cite{Aflalo2015pnas}\label{lm:equiv_problems}
Let $\mc{G}_1$ and $\mc{G}_2$ be two asymmetric and isomorphic graphs. Then, the problems \eqref{eq:graph_matching} and \eqref{eq:RGM} are equivalent.
\end{Lemma}
\begin{Remark}\label{lm:equiv_problems_perturbed}
In some circumstances, we might deal with isomorphic graph matching in the presence of noise.
 Let $\mc{G}_1$ and $\mc{G}_2$ be isomorphic and asymmetric graphs with spectrum radius $\sigma=\max_{i}|\lambda_i|$, related by the unique isomorphism $\m{\Pi}^*$. The adjacency matrix $\m{A}(\mc{G}_1)$ has the spectrum gap $\min_{i\neq j}|\lambda_i-\lambda_j|>\delta$ and all eigenvectors satisfy $\epsilon<|\m{u}_i^\top\m{1}_n|<\frac{1}{\epsilon}$, for $\delta,\epsilon>0$.
Let $\tilde{\m{B}}$ be a perturbed adjacency matrix of $\m{B}(\mc{G}_2)$ such that $\tilde{\m{B}}=\m{B}+\rho\m{R}$, where $\m{R}$ is symmetric with $||\m{R}||_F\leq 1$, and $\rho\leq \min\{\sqrt{2}\sigma,\frac{\delta^2\epsilon^4}{12\sigma n^{1.5}}\}$. Then, the optimal solution $\m{P}^*$ of the problem $\m{P}^*=\underset{\m{P}\in \mc{D}_n}{\text{argmin}}||\m{P}\m{A}-\tilde{\m{B}}\m{P}||_F^2,$ is unique and satisfies $||\m{P}^*-\m{\Pi}^*||_F<1/2$ \cite[Lemma 2]{Aflalo2015pnas}. The perturbation $\rho\m{R}$ characterizes the total adjacency disagreement that still allows the optimal permutation to be recovered. The conclusion of course does not actually depend on $\tilde{\m{B}}$ obtained by the addition of noise to a $\m{B}$ for which isomorphism holds, but simply on having a $\tilde{\m{B}}$ suitably close to $\m{B}$. By way of a side remark, we note that if $\m{B}$ contains zero entries, some entries of $\m{R}$ might need to be non-negative so that $\tilde{\m{B}}$ is a proper adjacency matrix.
\end{Remark}
Consequently, we can solve \eqref{eq:RGM} for $\m{P}^*$ and project it onto $\mc{P}$ to get $\hat{\m{\Pi}}$. If $\text{dis}_{\mc{G}_1\rightarrow\mc{G}_2}(\hat{\m{\Pi}})$ is small enough the graphs are isomorphic. The orthogonal projection $\m{P}^*$ onto $\mc{P}$ can be obtained by optimizing the Euclidean inner product
\begin{equation}\label{eq:permutation_projection}
\hat{\m{\Pi}}=\text{Proj}_{\mc{P}}\m{P}^*=\underset{\m{\Pi}\in \mc{P}}{\text{argmax}}~\text{tr}(\m\Pi^\top\m{P}^*),
\end{equation}
which is solvable in polynomial time using the Hungarian method \cite{Kuhn1955}. In addition, it follows from $||\m{P}^*-\m{\Pi}^*||_F<1/2$ in the preceding remark that $|(\m{P}^*)_{ij}-(\m{\Pi})_{ij}^*|<1/2$, for all $i,j$ entries. Thus, the projection $\text{Proj}_{\mc{P}}\m{P}^*$ can be simply obtained by rounding up/down the entries of $\m{P}^*$ to the nearest integer numbers in $\{0,1\}$, i.e., $\text{argmin}_{x\in \{0,1\}}|(\m{P}^*)_{ij}-x|$ entry-wise for every $i,j$. 

In the absence of noise, the optimal solution of \eqref{eq:RGM} is identical to $\m{\Pi}^*$, and hence a distributed algorithm with asymptotic stability property can compute $\m{\Pi}^*$ as $t\rightarrow \infty$. Further, such a permutation can be obtained after a finite time using projection when $||\m{P}(t)-\m{\Pi}^*||_F<1/2$, as discussed above.
In summary, there are two steps in solving the (GM) problem as summarized in Algorithm \ref{alg:GM}.

\begin{algorithm}[t]
        \caption{Centralized Computation of Graph Matching.}
        \label{alg:GM}
        \begin{algorithmic}[1]
            \Require $\m{A}$, $\m{B}$ adjacency matrices of isomorphic and asymmetric graphs.
			\State Solve the convex optimization problem \eqref{eq:RGM} for $\m{P}^*$.
			\State Project $\m{P}^*$ onto $\mc{P}$.
            \State\Return $\hat{\m{\Pi}}$;
        \end{algorithmic}
    \end{algorithm}
\section{Distributed Optimization to Solve RGM}\label{sec:distributed_algorithm}
In this part, we formulate the RGM problem \eqref{eq:RGM} as a distributed optimization problem and propose a \textit{distributed} optimization over a multi-agent network of $n$ agents to solve it.
\subsubsection*{Multi-agent network}
We assume that each agent $i$ in a network of $n$ agents only knows $\m{a}_i$ and $\m{b}_i\in \mb{R}^{n}$, the $i$th column vectors of $\m{A}$ and $\m{B}$, respectively, and can exchange information with some neighboring agents. This exchange process itself, which effectively defines the way calculations determining the optimum $\m{P}^*$ are distributed,  can be modelled by a  graph.  To differentiate with the two graphs $\mc{G}_1$ and $\mc{G}_2$ to be matched, we denote the interaction graph of the agents as $\mc{H}=(\mc{I},\mc{E}_{\mc{H}})$, where $\mc{I}=\{1,\ldots,n\}$ and $\mc{E}_{\mc{H}}\subseteq \mc{I}\times \mc{I}$ denote the index set and edge set of the agents, respectively. When agents $i$ and $j$ are neighbors, i.e., $(i,j)\in \mc{E}_{\mc{H}}$, we associate with this edge an arbitrary positive weight $w_{ij}=w_{ji}>0$; when $(i,j)\not\in \mc{E}_{\mc{H}}$, $w_{ij}=w_{ji}=0$. The graph $\mc{H}$ is assumed to be undirected and connected. 
\subsection{Geometric interpretation of the pseudo-stochastic matrix set $\mc{D}_n$ and projection operator}\label{subsec:geometrical_interp}
The sum of the elements in any row vector of a matrix in $\mc{D}_n=\{\m{P}:\m{P}\m{1}_n=\m{1}_n\}$ is one. Thus, pseudo-stochastic matrices contain $n$ rows which, if each row vector is considered as a point in the $n$-dimensional Cartesian space, belong to a hyperplane in $\mb{R}^{n}$, i.e., the plane $\sum_{k=1}^nx_k=1$, where $x_k,k=1,\ldots,n$ are the coordinates of a point vector $\m{x}\in \mb{R}^n$. Let $\mc{S}_n$ denote this plane. Then, the normal vector of the hyperplane $\mc{S}_n$ is $\m{1}_n$ (See Fig. \ref{fig:interpret_constrained_set}). For simplicity, we say a row of a matrix $\m{X}\in \mc{D}_n$ belongs to the hyperplane $\mc{S}_n$, when there is no risk of confusion.

The orthogonal projection of a vector $\m{v}\in \mb{R}^n$ onto the hyperplane $\mc{S}_n$ is given as 
\begin{equation}\label{eq:orthogonal_project}
\text{Proj}_{\mc{S}_n}(\m{v}):=(\m{I}_n-({1}/{n})\m{1}_n\m{1}_n^\top)\m{v}.
\end{equation}
Given a $n\times n$ matrix $\m{V}$, we denote by $\text{Proj}_{\mc{D}_n}(\m{V})$ the orthogonal projection of $\m{V}$ onto $\mc{D}_n$, i.e., 
\begin{equation}\label{eq:projection_of_matrix}
\text{Proj}_{\mc{D}_n}(\m{V}):=\m{V}(\m{I}_n-({1}/{n})\m{1}_n\m{1}_n^\top),
\end{equation} 
which consists of $n$ orthogonal projections of the $n$ corresponding row vectors of $\m{V}$ onto $\mc{S}_n$.
The matrix $\m{V}$ is said to be \textit{parallel} to $\mc{D}_n$ if $\text{Proj}_{\mc{D}_n}(\m{V})=\m{V}$, and is \textit{orthogonal} to $\mc{D}_n$ if $\text{Proj}_{\mc{D}_n}(\m{V})=\m{0}$.
\begin{Lemma}\label{lm:projection_property}
Let an arbitrary vector $\m{u}\in \mb{R}^n$ and any two points $\m{x},\m{y}\in \mc{S}_n$. Then, there holds
\begin{equation}
(\m{x}-\m{y})^\top\text{Proj}_{\mc{S}_n}(\m{u})=(\m{x}-\m{y})^\top\m{u}.
\end{equation}
\end{Lemma}
\begin{proof}
We have
\begin{align*}
(\m{x}-\m{y})^\top\text{Proj}_{\mc{S}_n}(\m{u})&=(\m{x}-\m{y})^\top(\m{I}_n-\frac{1}{n}\m{1}_n\m{1}_n^\top)\m{u}\\
&=(\m{x}-\m{y})^\top\m{u}-\frac{1}{n}(\m{x}-\m{y})^\top\m{1}_n\m{1}_n^\top\m{u}\\
&=(\m{x}-\m{y})^\top\m{u},
\end{align*}
where the last equality follows from $(\m{x}-\m{y})\perp\m{1}_n$, for any two points $\m{x},\m{y}\in \mc{S}_n$.
\end{proof}
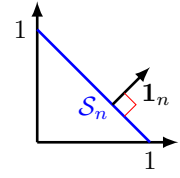
\begin{figure}[t]
\centering
\begin{tikzpicture}
\node (py) at (0,2.) [label=below left:$$]{};
\node (px) at (2.,0.) [label=below:$$]{};
\draw[{line width=1pt}] (0,0) [frame] -- (px) node [pos =0.8, yshift=0ex,below] {$1$}; 
\draw[{line width=1pt}] (0,0) [frame] -- (py) node [pos =0.8, yshift=0ex,left] {$1$};
\draw[{line width=1pt},blue] (0,1.5) -- (1.5,0) node [pos =0.5, yshift=-0.1ex, below] {$\mc{S}_n$};
\draw[{line width=1pt}] (1,0.5) -- (1.5,1.)[frame] node [pos =0.5, yshift=-0.7ex, right] {$\m{1}_n$};
\draw [right angle symbol={1,0.5}{1.5,1.}{1.5,0}, red];
\end{tikzpicture}
\caption{Geometric interpretation of the pseudo-stochastic matrix set $\mc{D}_n$. A point, whose coordinates are the elements of any row vector of a matrix in $\mc{D}_n$, lies in the hyperplane $\mc{S}_n$ in the $n$-dimensional space.}
\label{fig:interpret_constrained_set}
\end{figure}
\subsection{GM as a distributed optimization problem}
Assume that each agent $i\in \mc{I}$ knows $\m{a}_i$ and $\m{b}_i$ and holds a local estimate of the common optimal matrix $\m{P}^*$ of the problem (RGM), denoted by $\m{P}_i\in \mb{R}^{n\times n}$. The agents cooperatively estimate $\m{P}^*$ such that $\m{P}_i\rightarrow \m{P}^*,\forall i\in \mc{I}$, as $t\rightarrow \infty$. In the case of exact matching, due to $\m{A}=\m{A}^\top$ the optimal solution to the problem \eqref{eq:RGM} is equivalent to finding a matrix $\m{P}^*\in \mc{D}_n$ satisfying the system of equations
\begin{equation}\label{eq:reformulated_RGM}
\begin{cases}
\m{P}^*\m{a}_i=\m{y}_i^*,~i=1,\ldots,n,\\
\m{b}_i^\top\m{P}^*={\m{z}_i^*}^\top,~i=1,\ldots,n,\\
\m{Y}^*\triangleq [\m{y}_1^*,\ldots,\m{y}_n^*]=\m{Z}^*\triangleq\text{col}({\m{z}_1^*}^\top,\ldots,{\m{z}_n^*}^\top),\\
\m{P}^*\in \mc{D}_n,
\end{cases}
\end{equation}
where $\m{y}_i^*,\m{z}_i^*\in\mb{R}^{n}$ for all $i\in\mc{I}$. As a result, in addition to $\m{P}_i$ each agent $i\in\mc{I}$ maintains two variables $\m{y}_i$ and $\m{z}_i$. The third relation in the preceding equation is introduced to impose equality constraints involving $\m{y}_i$ and $\m{z}_i$, i.e., $\m{Y}=\m{Z}$. While the consensus constraints $\m{P}_i=\m{P}_j,\forall i,j\in \mc{I},$ can be easily treated, i.e., through distributed averaging, the coupling constraint $\m{Y}=\m{Z}$ is not separable to each agent $i,\forall i\in \mc{I}$ since it only knows the $i$th column of $\m{Y}$ and the $i$th row of $\m{Z}$. To deal with such coupling constraint, we use the following transformation, for all $i\in \mc{I}$ (see e.g., \cite{XZeng2019tac, Deng2019}):
\begin{align*}
\m{Y}&=\m{Z}\Longleftrightarrow \exists\{\m{K}_i\}_{i=1}^n,~\text{s.t.}\\
[\m{Y}]_i&-[\m{Z}]^i-\sum_{j=1}^nw_{ij}(\m{K}_i-\m{K}_j)=\m{0}_{n\times n}, \numberthis \label{eq:coupling_constraint_transf}
\end{align*}
where $$[\m{Y}]_i:=[\m{0}_{n\times(i-1)},\m{y}_i,\m{0}_{n\times(n-i)}]\in \mb{R}^{n\times n}$$ is an $n\times n$ matrix whose $i$th column is $\m{y}_i$ and other columns are zeros, 
$$[\m{Z}]^i:=\text{col}(\m{0}_{(i-1)\times n},\m{z}_i^\top,\m{0}_{(n-i)\times n})\in \mb{R}^{n\times n}$$ is a matrix of all zero row vectors except the $i$th row is being $\m{z}_i^\top$,
and $\m{K}_i\in \mb{R}^{n\times n},i=1,\ldots,n$ are used to compensate for the inconsistencies between $[\m{Y}]_i$ and $[\m{Z}]^i$. Note that $\m{Y}=\sum_{i=1}^n[\m{Y}]_i$ and $\m{Z}=\sum_{i=1}^n[\m{Z}]^i$. In addition, by summing up both sides of \eqref{eq:coupling_constraint_transf} over $i$ from $1$ to $n$ we obtain $\m{Y}=\m{Z}$.

Let $\m{X}=\text{col}(\m{P}_1,\ldots,\m{P}_n)$ be the stack matrix of all local estimates of $\m{P}^*$ at some intermediate point in the execution of the algorithm. Using the relations \eqref{eq:reformulated_RGM} and \eqref{eq:coupling_constraint_transf}, we reformulate the problem \eqref{eq:RGM} as a distributed optimization problem subject to a global set constraint and equality constraints, over the multi-agent network:
\begin{equation}\label{eq:distributed_GM_prob}
\begin{cases}
\underset{\m{X},\m{Y},\m{Z},\m{K}}{\text{argmin}}~\frac{1}{2}\sum_{i=1}^n||\m{P}_i\m{a}_i-\m{y}_i||^2\\
\text{s.t.}~\m{b}_i^\top\m{P}_i=\m{z}_i^\top,~\m{P}_i\in \mc{D}_n,\\
\sum_{j=1}^nw_{ij}(\m{P}_i-\m{P}_j)=\m{0}_{n\times n},\\
[\m{Y}]_i-[\m{Z}]^i-\sum_{j=1}^nw_{ij}(\m{K}_i-\m{K}_j)=\m{0}_{n\times n},
\end{cases}
\end{equation}
for all $i\in \mc{I}$, where $\m{K}=\text{col}(\m{K}_1,\ldots,\m{K}_n)$. In addition, $\m{P}_i\m{a}_i$ is the agent $i$'s estimate of the $i$th column vector of $\m{P}^*\m{A}$ and is assigned to the vector $\m{y}_i$, while $\m{b}_i^\top\m{P}_i$ is its estimate of the $i$th row of $\m{BP}^*$ and is stored at the row vector $\m{z}_i^\top$. These two vectors satisfy the last coupling equality constraint in \eqref{eq:distributed_GM_prob}. Consequently, the agents cooperatively learn the graph matching using knowledge of $(\m{a}_i,\m{b}_i)$ and the auxiliary state variables $(\m{P}_i,\m{y}_i,\m{z}_i,\m{K}_i)$, for all $i\in \mc{I}$. In the sequel, we develop a projected multi-agent optimization algorithm to solve the constrained optimization problem \eqref{eq:distributed_GM_prob}.
\subsection{Distributed learning scheme}
Consider the Lagrangian function of the problem \eqref{eq:distributed_GM_prob}
\begin{align*}
\mc{L}_1&=\frac{1}{2}\sum_{i=1}^n||\m{P}_i\m{a}_i-\m{y}_i||^2+\sum_{i=1}^n\langle\boldsymbol{\lambda}_i^\top,\m{b}_i^\top\m{P}_i-\m{z}_i^\top\rangle\\
&
+\sum_{i=1}^n\langle\Theta_i,\sum_{j=1}^nw_{ij}(\m{P}_i-\m{P}_j)\rangle\\
&+\sum_{i=1}^n\langle\Upsilon_i,[\m{Y}]_i-[\m{Z}]^i-\sum_{j=1}^nw_{ij}(\m{K}_i-\m{K}_j)\rangle,
\end{align*}
where $\boldsymbol{\lambda}_i\in \mb{R}^n$, $\Theta_i\in \mb{R}^{n\times n}$ and $\Upsilon_i\in \mb{R}^{n\times n}$ are Lagrange multipliers of agent $i$ associated with the equality constraints in problem \eqref{eq:distributed_GM_prob}, for all $i\in \mc{I}$. We first define
\begin{align*}
\boldsymbol{\lambda}^\top&=\text{col}(\boldsymbol{\lambda}_1^\top,\ldots,\boldsymbol{\lambda}_n^\top)\in \mb{R}^{nn},\\
\Theta&=\text{col}(\Theta_1,\ldots,\Theta_n)\in \mb{R}^{nn\times n},\\
\Upsilon&=\text{col}(\Upsilon_1,\ldots,\Upsilon_n)\in \mb{R}^{nn\times n},\\
\m{Q}&=\text{col}(\m{X},\m{Y},\m{Z},\m{K},\boldsymbol{\lambda},\Theta,\Upsilon).
\end{align*}
In addition, let $\m{Q}^*=\text{col}(\m{X}^*,\m{Y}^*,\m{Z}^*,\m{K}^*,\boldsymbol{\lambda}^*,\Theta^*,\Upsilon^*)$ be an optimal solution of \eqref{eq:distributed_GM_prob}. Then, we have the following lemma.
\begin{Lemma}
The optimal states $(\m{P}_i^*,\m{y}_i^*,\m{z}_i^*,\m{K}_i^*)$ and the Lagrange multipliers (or the dual optimal variables) $(\boldsymbol{\lambda}_i^*,\Theta_i^*,\Upsilon_i^*)$, for all $i\in \mc{I}$, satisfy the following necessary and sufficient optimality condition:


\begin{subnumcases}{\label{eq:KKT_condition}}
  \begin{split}&\text{Proj}_{\mc{D}_n}\Big\{(\m{P}_i^*\m{a}_i-\m{y}_i^*)\m{a}_i^\top+\m{b}_i{\boldsymbol{\lambda}_i^*}^\top+ \\
&\qquad\qquad+\textstyle\sum_{j=1}^nw_{ij}(\Theta_i^*-\Theta_j^*)\Big\}=\m{0} \label{eq:KKT_Pi}
\end{split}\\
\nabla_{\boldsymbol{\lambda}_i^\top}\mc{L}_1(\m{Q}^*)=\m{b}_i^\top\m{P}_i^*-{\m{z}_i^*}^\top=\m{0}\label{eq:KKT_lamda}\\
\nabla_{\Theta_i}\mc{L}_1(\m{Q}^*)=\textstyle\sum_{j=1}^nw_{ij}(\m{P}_i^*-\m{P}_j^*)=\m{0}\label{eq:KKT_Theta}\\
\begin{split}\nabla_{\Upsilon_i}\mc{L}_1(\m{Q}^*)&=[\m{Y}^*]_i-{[\m{Z}^*]^i}\\
&\qquad-\textstyle\sum_{j=1}^nw_{ij}(\m{K}_i^*-\m{K}_j^*)=\m{0}\label{eq:KKT_Upsilon}\end{split}\\
\nabla_{\m{y}_i}\mc{L}_1(\m{Q}^*)=-(\m{P}_i^*\m{a}_i-\m{y}_i^*)+(\Upsilon_i^*)_{i}^\text{C}=\m{0}\label{eq:KKT_yi}\\
\nabla_{\m{z}_i^\top}\mc{L}_1(\m{Q}^*)=-{\boldsymbol{\lambda}_i^*}^\top-(\Upsilon_i^*)_i^\text{R}=\m{0}\label{eq:KKT_zi}\\
\nabla_{\m{K}_i}\mc{L}_1(\m{Q}^*)=-\textstyle\sum_{j=1}^n(\Upsilon_i^*-\Upsilon_j^*)=\m{0}.\label{eq:KKT_Ki}
\end{subnumcases}
\end{Lemma}

\begin{proof}
$(\m{P}_i^*,\m{y}_i^*,\m{z}_i^*,\m{K}_i^*)$ is the optimal point of \eqref{eq:distributed_GM_prob} if and only if there exist $(\boldsymbol{\lambda}_i^*,\Theta_i^*,\Upsilon_i^*)$, for all $i\in \mc{I}$, such that the following relations hold (by using a similar argument as in \cite[Thm. 3.34]{Ruszczynski2004}):
\begin{align*}
  &(\m{P}_i^*\m{a}_i-\m{y}_i^*)\m{a}_i^\top+\m{b}_i{\boldsymbol{\lambda}_i^*}^\top+ \sum_{j=1}^nw_{ij}(\Theta_i^*-\Theta_j^*)\\
  &\qquad\qquad\qquad\qquad\qquad+\m{N}_{\mc{D}_n}(\m{P}_i^*)=\m{0}, \numberthis\label{eq:optimal_condition_normal}\\
&\eqref{eq:KKT_lamda}-\eqref{eq:KKT_Ki},~\text{for all~} i=1,\ldots,n,
\end{align*}
for an $n\times n$ matrix $\m{N}_{\mc{D}_n}(\m{P}_i^*)=\m{x}\m{1}_n^\top$ with an arbitrary $\m{x}\in \mb{R}^n$, whose row vectors are in $\text{span}(\m{1}_n^\top)$ or, i.e., orthogonal to the hyperplane $\mc{S}_n$ defined in Section \ref{subsec:geometrical_interp}. Since the subspaces $\mc{D}_n$ and $\m{N}_{\mc{D}_n}$ are orthogonal in the sense that $\langle \m{\Delta},\m{x}\m{1}_n^\top \rangle=0$ for any $\m{\Delta}\in \mb{R}^{n\times n}$ parallel to $\mc{D}_n$, \eqref{eq:optimal_condition_normal} is equivalent to \eqref{eq:KKT_Pi}.
\end{proof}

Consider the augmented Lagrangian function:
\begin{equation}\label{eq:augmented_Lag}
\begin{split}
&\mc{L}_2=\mc{L}_1+\frac{1}{2}\sum_{i=1}^n\langle \m{P}_i,\sum_{j=1}^nw_{ij}(\m{P}_i-\m{P}_j)\rangle\\
&+\frac{1}{2}\sum_{i=1}^n||\m{b}_i^\top\m{P}_i-\m{z}_i^\top||^2-\frac{1}{2}\sum_{i=1}^n\langle \Upsilon_i,\sum_{j=1}^nw_{ij}(\Upsilon_i-\Upsilon_j)\rangle,
\end{split}
\end{equation} where the last three terms are augmented terms. The additional (quadratic) augmented terms in $\mc{L}_2$, which vanish at an optimal solution $\m{Q}^*$ of problem \eqref{eq:distributed_GM_prob} due to Eq. \eqref{eq:KKT_condition}, are used to impose further constraints in the positively invariant set of the system \eqref{eq:distributed_optimization}, as will be shown in Eq. \eqref{eq:dotV_simpler} below. Primal-dual gradient methods using augmented Lagrangian functions can be found in \cite{Gharesifard2014tac,XZeng2019tac, TYang2019}. We employ the primal-dual gradient method for the problem \eqref{eq:distributed_GM_prob} that evolves on the manifold $\mc{D}_n$, as described in what follows.

\begin{algorithm}[t]
        \caption{Distributed Algorithm for Solving \eqref{eq:distributed_GM_prob}.}
        \label{alg:distributed_RGM}
        \begin{algorithmic}[1]
        	\State \textbf{Initialize:} $\m{P}_i(0)\in \mc{D}_n,\m{y}_i(0)\in \mb{R}^n,\m{z}_i(0)\in \mb{R}^{n},\m{K}_i(0)\in\mb{R}^{n\times n},$ $\boldsymbol{\lambda}_i(0)\in \mb{R}^n$, $\Theta_i(0)\in \mb{R}^{n\times n}$ and $\Upsilon_i(0)\in \mb{R}^{n\times n},\forall i\in \mc{I}$.
        	\State \textbf{Update rules:}
        	 \begin{subnumcases}{\label{eq:distributed_optimization}}
        	 \begin{split}&\dot{\m{P}}_i(t)=\text{Proj}_{\mc{D}_n}\Big\{-(\m{P}_i\m{a}_i-\m{y}_i)\m{a}_i^\top-\m{b}_i\boldsymbol{\lambda}_i^\top\\
        	 &-\sum_{j=1}^nw_{ij}(\Theta_i-\Theta_j)-\sum_{j=1}^nw_{ij}(\m{P}_i-\m{P}_j)\\
        	 &-\m{b}_i(\m{b}_i^\top\m{P}_i-\m{z}_i^\top)\Big\}\label{eq:Pi_dot}
\end{split}\\
			\boldsymbol{\dot{\lambda}}_i^\top(t)=\m{b}_i^\top\m{P}_i-\m{z}_i^\top\label{eq:lambda_i_dot}\\
        	 \dot{\Theta}_i(t)=\textstyle\sum_{j=1}^nw_{ij}(\m{P}_i-\m{P}_j)\label{eq:Theta_i_dot}\\
        	\begin{split} 
        	\dot{\Upsilon}_i(t)&=[\m{Y}]_i-{[\m{Z}]^i}-\textstyle\sum_{j=1}^nw_{ij}(\m{K}_i-\m{K}_j)\\
			& \qquad-\textstyle\sum_{j=1}^nw_{ij}(\Upsilon_i-\Upsilon_j)       	
        	\label{eq:Upsilon_i_dot}
			\end{split}\\
			\dot{\m{y}}_i(t)=(\m{P}_i\m{a}_i-\m{y}_i)-(\Upsilon_i)_i^\text{C} \label{eq:y_i_dot}\\
			\dot{\m{z}}_i^\top(t)=\boldsymbol{\lambda}_i^\top+(\Upsilon_i)_i^\text{R}+\m{b}_i^\top\m{P}_i-\m{z}_i^\top \label{eq:z_i_dot}\\
			\dot{\m{K}}_i(t)=\textstyle\sum_{j=1}^nw_{ij}(\Upsilon_i-\Upsilon_j) \label{eq:K_i_dot}
        	 \end{subnumcases}
        	 for all $i=1,\ldots,n.$
        \end{algorithmic}
\end{algorithm}
We propose the following distributed algorithm to solve the problem \eqref{eq:distributed_GM_prob} based on the saddle-point dynamics of the augmented Lagrangian function $\mc{L}_2$.
\begin{align*}
\dot{\m{P}}_i&=\text{Proj}_{\mc{D}_n}(-\nabla_{\m{P}_i}\mc{L}_2),\\
\dot{\phi}_i&=-\nabla_{\phi_i}\mc{L}_2, ~\text{for} ~\phi_i\in \{\m{y}_i,\m{z}_i^\top,\m{K}_i\},\\
\dot{\psi}_i&=\nabla_{\psi_i}\mc{L}_2, ~\text{for~} \psi_i\in \{\boldsymbol{\lambda}_i^\top,\Theta_i,\Upsilon_i\},~\forall i\in\mc{I}.
\end{align*}
The continuous-time dynamics of the primal and dual variables are explicitly given in \eqref{eq:distributed_optimization} in Algorithm \ref{alg:distributed_RGM}. A projection-like gradient algorithm used to solve a linear algebraic equation in \cite{Anderson2016naco} is not straightforwardly applicable for \eqref{eq:distributed_GM_prob} due to the presence of coupling constraints in \eqref{eq:distributed_GM_prob}.
\subsection{Stability Analysis}
Assume that two asymmetric graphs $\mc{G}_1=(\mc{V},\m{A})$ and $\mc{G}_2=(\mc{V},\m{B})$ are isomorphic. We first show that $\m{P}_i$ are well-defined for all $i\in \mc{I}$.
\begin{Lemma}\label{lm:Pi_in_Dn}
Suppose $\mc{G}_1$ and $\mc{G}_2$ are two isomorphic and asymmetric graphs. Then, under the update law \eqref{eq:distributed_optimization}, for any initial matrix $\m{P}_i(0)\in \mc{D}_n$, $\m{P}_i(t)$ lies in the convex set $\mc{D}_n$, or equivalently, rows of $\m{P}_i(t)$ lie in the hyperplane $\mc{S}_n$, for all $i\in\mc{I}$, for all time $t\geq 0$.
\end{Lemma}
\begin{proof}
Consider $\frac{d}{dt}(\m{P}_i\m{1}_n)=\dot{\m{P}}_i\m{1}_n=\text{Proj}_{\mc{D}_n}(\Delta_i)\m{1}_n=\Delta_i(\m{I}_n-\frac{1}{n}\m{1}_n\m{1}_n^\top)\m{1}_n=\m{0}$, where $\Delta_i$ is the expression inside the projection operator corresponding to agent $i$ in \eqref{eq:Pi_dot}. It follows that $\m{P}_i\m{1}_n$ is time invariant under \eqref{eq:distributed_optimization} and hence $\m{P}_i\m{1}_n=\m{P}_i(0)\m{1}_n=\m{1}_n$, for all $i\in\mc{I}$, for all time $t\geq 0$. This completes the proof.
\end{proof}
Then, in the light of Lemma \ref{lm:equiv_problems}, the optimal solution of the problem \eqref{eq:distributed_GM_prob} has a unique $\m{\Pi}^*$.
\begin{Lemma}\label{lm:unique_minimizer}
Let $\mc{G}_1$ and $\mc{G}_2$ be two asymmetric and isomorphic graphs, related by an unique permutation $\m{\Pi}^*\in \mc{P}$. Then, $(\m{X}^*,\m{Y}^*,\m{Z}^*,\m{K}^*)$ is an optimal solution of the problem \eqref{eq:distributed_GM_prob} if and only if there exist $\boldsymbol{\lambda}^*\in \mb{R}^{n^2},\Theta^*\in \mb{R}^{n^2\times n},\Upsilon^*\in \mb{R}^{n^2\times n}$, such that $\m{Q}^*$ is an equilibrium point of \eqref{eq:distributed_optimization}. Moreover, such $\m{P}_i^*=\m{\Pi}^*,\forall i\in \mc{I}$.
\end{Lemma}
\begin{proof}
The proof follows from the necessary and sufficient condition for optimality \eqref{eq:KKT_condition} and Lemma \ref{lm:equiv_problems}.
\end{proof}

\begin{Theorem}\label{thm:global_asymp_convergence}
Let $\mc{G}_1$ and $\mc{G}_2$ be two isomorphic and asymmetric graphs. Then, every trajectory $\m{Q}(t)$ of the system \eqref{eq:distributed_optimization} in Algorithm \ref{alg:distributed_RGM} with an initial condition $\m{Q}(0)$, converges globally asymptotically to an equilibrium of \eqref{eq:distributed_optimization}. In addition, $\m{P}_i\rightarrow\m{\Pi}^*$ as $t\rightarrow \infty$, and the orthogonal projection of $\m{P}_i$ onto $\mc{P}$ is identical to $\m{\Pi}^*$ after a finite time $T>0$ when $||\m{P}_i(T)-\m{\Pi}^*||_F<1/2$, for all $i=1,\ldots,n$.
\end{Theorem}
\begin{proof}
It can be verified that the optimal solution $\m{Q}^*=\text{col}(\m{X}^*,\m{Y}^*,\m{Z}^*,\m{K}^*,\boldsymbol{\lambda}^*,\Theta^*,\Upsilon^*)$ satisfying the optimality condition \eqref{eq:KKT_condition} is an equilibrium of the system \eqref{eq:distributed_optimization}. Consider the Lyapunov function $V=\frac{1}{2}||\m{Q}-\m{Q}^*||_F^2$ which is positive definite and radially unbounded. The time derivative of $V$ along the trajectory of \eqref{eq:distributed_optimization} is given as
\begin{align*}
&\dot{V}=\text{tr}\big(\dot{\m{Q}}^\top(\m{Q}-\m{Q}^*)\big)\\
&=\text{tr}\Big\{ \sum_{i=1}^n\dot{\m{P}}_i^\top(\m{P}_i-\m{P}_i^*)+\sum_{i=1}^n\dot{\m{y}}_i^\top(\m{y}_i-\m{y}_i^*)\\
&+\sum_{i=1}^n\dot{\m{z}}_i(\m{z}_i-\m{z}_i^*)^\top+ \sum_{i=1}^n\dot{\m{K}}_i^\top(\m{K}_i-\m{K}_i^*)\\
&+\sum_{\psi_i=\boldsymbol{\lambda}_i^\top,\Theta_i,\Upsilon_i}\sum_{i=1}^n\dot{\psi}_i^\top(\psi_i-\psi_i^*)\Big\}. \numberthis \label{eq:dotV}
\end{align*}
Let $\Delta_i\in \mb{R}^{n\times n}$ be the expression inside the projection operator corresponding to agent $i$ in \eqref{eq:Pi_dot}. Then, it follows from Lemmas \ref{lm:projection_property} and \ref{lm:Pi_in_Dn} and the properties of the trace function that, for all $i\in \mc{I}$,
\begin{align*}
\text{tr}\big\{\dot{\m{P}}_i^\top(\m{P}_i-\m{P}_i^*)\big\}&=\text{tr}\big\{(\m{P}_i-\m{P}_i^*)\dot{\m{P}}_i^\top\big\}\\
&\stackrel{\eqref{eq:Pi_dot}}{=}\text{tr}\big\{(\m{P}_i-\m{P}_i^*)\big[\text{Proj}_{\mc{D}_n}(\Delta_i)\big]^\top\big\}\\
&=\sum_{j=1}^n(\m{P}_i-\m{P}_i^*)^j\text{Proj}_{\mc{S}_n}\big({(\Delta_i)^j}^\top\big)
\\
&\stackrel{Lem. \ref{lm:projection_property}}{=}\sum_{j=1}^n(\m{P}_i-\m{P}_i^*)^j{(\Delta_i)^j}^\top
\\
&=\text{tr}\big\{(\m{P}_i-\m{P}_i^*)\Delta_i^\top\big\}\\
\Leftrightarrow \text{tr}\big\{\dot{\m{P}}_i^\top(\m{P}_i-\m{P}_i^*)\big\}&=\text{tr}\big\{\Delta_i^\top(\m{P}_i-\m{P}_i^*)\big\},\numberthis \label{eq:dotV_projection_relation}
\end{align*}
where $(\cdot)^j$ denotes the $j$th row vector of the associated matrix.

\subsection*{Step 1: Negative semidefiniteness of $\dot{V}(t)$}
By using the optimality condition \eqref{eq:KKT_condition} and the relation \eqref{eq:dotV_projection_relation}, it can be shown that $\dot{V}$ in \eqref{eq:dotV} along the trajectory of \eqref{eq:distributed_optimization} has the following more concise expression (see Appendix \ref{app:negative_dotV}):
\begin{align*}
\dot{V}&=-\sum_{i=1}^n\Big\{||(\m{P}_i-\m{P}_i^*)\m{a}_i-(\m{y}_i-\m{y}_i^*)||^2\\
&\qquad+||\m{b}_i^\top\m{P}_i-\m{z}_i^\top)||^2\Big\}\\
&-\sum_{(i,j)\in \mc{E}_{\mc{H}}}w_{ij}\Big\{||\m{P}_i-\m{P}_j||_F^2+||\Upsilon_i-\Upsilon_j||_F^2\Big\},\numberthis \label{eq:dotV_simpler}
\end{align*}
which is negative semidefinite. Consequently, $V(t)$ is bounded, i.e., $V(t)\leq V(0)$. Since $V(\m{Q})$ is radially unbounded, every level set $\Omega_c=\{\m{Q}:V(\m{Q})\leq c\}$ with a positive $c\in \mb{R}$, is a compact, positively invariant set. It follows that any trajectory $\m{Q}(t)$ of the system \eqref{eq:distributed_optimization} is bounded and converges to the largest invariant set that contains $\m{Q}$ such that $\dot{V}(\m{Q})=0$ due to the LaSalle's invariance principle (Lemma \ref{lm:LaSalle_principle}).
\subsection*{Step 2: Globally Asymptotic convergence of the optimal solution}
Let the invariant set $\mc{S}_{Q}\triangleq \{\m{Q}:\dot{V}(\m{Q})=0\}=\{\m{Q}:\m{P}_i=\m{P}_j,\m{P}_i\in \mc{D}_n,\Upsilon_i=\Upsilon_j,(\m{P}_i-\m{P}_i^*)\m{a}_i=\m{y}_i-\m{y}_i^*,\m{b}_i^\top\m{P}_i=\m{z}_i^\top,\forall i,j\in \mc{I},i\neq j\}$. We consider a solution trajectory of \eqref{eq:distributed_optimization} that satisfies $\bar{\m{Q}}(t)\in \mc{S}_Q$. Then, there holds:
\begin{enumerate}[i)]
\item $\boldsymbol{\dot{\bar\lambda}}_i(t)=\m{0}, \dot{\bar\Theta}_i=\m{0}$ and $\dot{\bar{\m{K}}}_i=\m{0}$ for all $i=1,\ldots,n$. Thus, $\boldsymbol{\bar\lambda}_i,\bar\Theta_i,\m{\bar K}_i$ are constants for all $i=1,\ldots,n$.
\item $\Delta_i:=-(\bar{\m{P}}_i\m{a}_i-\bar{\m{y}}_i)\m{a}_i^\top-\m{b}_i\boldsymbol{\bar\lambda}_i^\top-\sum_{j=1}^nw_{ij}(\bar\Theta_i-\bar\Theta_j)-\sum_{j=1}^nw_{ij}(\bar{\m{P}}_i-\bar{\m{P}}_j)-\m{b}_i(\m{b}_i^\top\bar{\m{P}}_i-\m{z}_i^\top)=-(\m{P}_i^*\m{a}_i-\m{y}_i^*)\m{a}_i^\top-\m{b}_i\boldsymbol{\bar\lambda}_i^\top-\sum_{j=1}^nw_{ij}(\bar\Theta_i-\bar\Theta_j)$ is constant because $\boldsymbol{\bar\lambda}_i,\bar\Theta_i$ are constants. As a result, $\dot{\bar{\m{P}}}_i=\text{Proj}_{\mc{D}_n}(\Delta_i)$ is constant. Since $\bar{\m{P}}_i,\forall i\in \mc{I}$ is bounded, we have $\dot{\bar{\m{P}}}_i=\m{0},$ for all $i \in \mc{I}$. It follows that $\bar{\m{P}}_i,\forall i\in \mc{I}$ are constants.
\item $\bar{\m{z}}_i^\top=\m{b}_i^\top\bar{\m{P}}_i$ is constant and hence $\dot{\bar{\m{z}}}_i^\top=\m{0}$ for all $i=1,\ldots,n$. Similarly, $\bar{\m{y}}_i=(\bar{\m{P}}_i-\m{P}_i^*)\m{a}_i+\bar{\m{y}}_i^*$ is constant and consequently $\dot{\bar{\m{y}}}_i=\m{0}$, for all $i=1,\ldots,n$. As a result, $\dot{\bar\Upsilon}_i=[\bar{\m{Y}}]_i-{[\bar{\m{Z}}]^i}-\sum_{j=1}^nw_{ij}(\bar{\m{K}}_i-\bar{\m{K}}_j)$ is constant. It follows that $\dot{\bar\Upsilon}_i=\m{0}$ due to boundedness of ${\bar\Upsilon}_i$, for all $i=1,\ldots,n$.
\end{enumerate}
It follows that the largest invariant set in $\mc{S}_Q$ contains only the equilibrium point of \eqref{eq:distributed_optimization}, which is globally asymptotically stable. 
It follows from Lemma \ref{lm:unique_minimizer} that any trajectory $(\m{X},\m{Y},\m{Z},\m{K})$ with initial condition $\m{Q}(0)$ converges globally and asymptotically to an optimal solution of problem \eqref{eq:distributed_GM_prob}. Further, $\m{P}_i\rightarrow \m{\Pi}^*$ as $t\rightarrow \infty$, for all $i=1,\ldots,n$, due to Lemma \ref{lm:unique_minimizer}. It follows that $||\m{P}^*-\m{\Pi}^*||_F\rightarrow 0$ asymptotically as $t\rightarrow \infty$, and the projection of $\m{P}_i$ onto the permutation set $\mc{P}$ is identical to $\m{\Pi}^*$ after some finite time $T>0$ such that $||\m{P}_i(T)-\m{\Pi}^*||_F<1/2$, for all $i\in \mc{I}$.
\end{proof}
\subsection{Convergence rate}
In this part, we will show that the convergence of the system \eqref{eq:distributed_optimization} is exponentially fast. To proceed, let $\m{P}_{\text{proj}}:=\m{I}_n-\frac{1}{n}\m{1}_n\m{1}_n^\top$, which is a constant matrix, be the projection matrix associated with the projection operation $\text{Proj}_{\mc{D}_n}$ in \eqref{eq:projection_of_matrix}. 
\begin{Theorem}
Let $\mc{G}_1$ and $\mc{G}_2$ are two isomorphic and asymmetric graphs. Then, the trajectory $\m{Q}(t)$ of the system \eqref{eq:distributed_optimization} in Algorithm \ref{alg:distributed_RGM} with an initial condition $\m{Q}(0)$ converges globally exponentially to an equilibrium of \eqref{eq:distributed_optimization}. In addition, $\m{P}_i\rightarrow\m{\Pi}^*$ exponentially fast as $t\rightarrow \infty$.
\end{Theorem}
\begin{proof}
The projected update law \eqref{eq:Pi_dot} can be simply written as
\begin{equation}
\dot{\m{P}}_i=\Delta_i\m{P}_{\text{proj}},
\end{equation}
where $\Delta_i$ is the expression inside the projection operator in \eqref{eq:Pi_dot}. It is noted that the preceding equation is a linear system, since $\text{Proj}_{\mc{D}_n}$ is just a linear operation.
By using the preceding equation and the following relations
\begin{align*}
(\m{I}_n\otimes \m{X}^\top)\text{vec}(\m{Y}^\top)&=\m{YX},\\
(\m{X}\otimes \m{I}_n)\text{vec}(\m{Y}^\top)&=\m{XY},
\end{align*}
for any two matrices $\m{X},\m{Y}$ of suitable dimensions, the system \eqref{eq:distributed_optimization} can be expressed as a linear invariant system of $\text{vec}(\m{Q}^\top)$. Then, the conclusion on exponential convergence follows from Lemma \ref{lm:exponen__converg_linear_system} and Theorem \ref{thm:global_asymp_convergence}.
\end{proof}
\begin{Remark}
Though the time $T>0$ associated with convergence into the ball around $\m{\Pi}^*$ has been shown to exist under the conditions in Theorem \ref{thm:global_asymp_convergence}, it is not straightforward to estimate even an upper bound for this convergence time nor to determine in the course of executing the algorithm that the time has been reached. However, the evolution of $\m{P}_i(t)$ with respect to time $t$ might give some hint on whether an optimal permutation has been obtained by an agent $i\in \mc{I}$. In particular, every entry $(\m{P}_i)_{kl}(t)$ will change relatively slowly after a sufficiently large time $T_1>0$ due to the exponential convergence of $\m{P}_i(t)$ to $\m{\Pi}^*$. In addition, $|(\m{P}_i)_{kl}(t)-(\m{\Pi})_{kl}^*|$ will also remain sufficiently small for all entries $k,l$, and for a proper projected permutation $\m{\Pi}^*=\text{Proj}_{\mc{P}}\m{P}_i(t)$, after a sufficiently large time.
\end{Remark}
\section{Simulation}\label{sec:sim}
Consider two weighed and connected graphs $\mc{G}_1,\mc{G}_2$ of six vertices given in Fig. \ref{fig:sim_graph}. In addition, the adjacency matrices associated with $\mc{G}_1,\mc{G}_2$ are explicitly given as
\begin{align}
\m{A}&=\left[\begin{smallmatrix}
0&1&0&0&0.95&0\\
    1&0&0.9&0&0.85&0\\
    0&0.9&0&1.5&0&0\\
    0&0&1.5&0&1.75&0\\
    0.95&0.85&0&1.75&0&0.8\\
    0&0&0&0&0.8&0
\end{smallmatrix}\right],\\
\m{B}&=\left[\begin{smallmatrix}
0&0&0.95&1.75&0.8&0.85\\
    0&0&0&1.5&0&0.9\\
    0.95&0&0&0&0&1\\
    1.75&1.5&0&0&0&0\\
    0.8&0&0&0&0&0\\
    0.85&0.9&1&0&0&0
\end{smallmatrix}\right].
\end{align}
It is noted that the two graphs are asymmetric and isomorphic.
The computation graph $\mc{H}$ of the multi-agent network is chosen to be identical to $\mc{G}_1$.
The vertex-to-vertex matchings between $\mc{G}_1$ and $\mc{G}_2$ are illustrated in Fig. \ref{fig:matching_result}. The optimal permutation matrix is given as
\begin{equation}
\m{\Pi}^*=\left[\begin{smallmatrix}
	 0     &0     &0     &0     &1     &0\\
     0     &0     &1     &0     &0     &0\\
     1     &0     &0     &0     &0     &0\\
     0     &0     &0     &1     &0     &0\\
     0     &0     &0     &0     &0     &1\\
     0     &1     &0     &0     &0     &0
\end{smallmatrix}\right].
\end{equation}
It is observed from Fig. \ref{fig:deltaPi} that $\m{P}_i\rightarrow \m{\Pi}^*$ asymptotically as $t\rightarrow \infty$, for every $i=1,\ldots,6$. In addition, the adjacency distortion $||\m{P}_1\m{A}-\m{B}\m{P}_1||_F^2$ converges to zero asymptotically as $t\rightarrow \infty$ (Fig. \ref{fig:distortion}).
\begin{figure}[t]
\centering
\begin{tikzpicture}[scale=1.]
\node[place] (4) at (1,0.) [label=right:$4$] {};
\node[place] (3) at (0,0) [label=left:$3$] {};
\node[place] (1) at (0,1.6) [label=left:$1$] {};
\node[place] (2) at (0,1.) [label=left:$2$] {};
\node[place] (5) at (1.,1.) [label=right:$5$] {};
\node[place] (6) at (1.,2) [label=right:$6$] {};
\node[] (g1) at (0.5,-0.2) [label=below:$\mathcal{G}_1$] {};

\draw (2) [line width=1pt] -- node [left] {} (3);
\draw (3) [line width=1pt] -- node [left] {} (4);
\draw (4) [line width=1pt] -- node [left] {} (5);
\draw (2) [line width=1pt] -- node [left] {} (5);
\draw (2) [line width=1pt] -- node [left] {} (1);
\draw (1) [line width=1pt] -- node [left] {} (5);
\draw (6) [line width=1pt] -- node [left] {} (5);

\node[place] (4') at (1+2.5,0.) [label=right:$4$] {};
\node[place] (2') at (0+2.5,0) [label=left:$2$] {};
\node[place] (3') at (1+2.5,1.6) [label=above:$3$] {};
\node[place] (6') at (0+2.5,1.) [label=left:$6$] {};
\node[place] (1') at (1.+2.5,1.) [label=right:$1$] {};
\node[place] (5') at (0.5+2.5,0.5) [label=right:$5$] {};
\node[] (g1) at (0.5+2.5,-0.2) [label=below:$\mathcal{G}_2$] {};

\draw (2') [line width=1pt] -- node [left] {} (6');
\draw (2') [line width=1pt] -- node [left] {} (4');
\draw (4') [line width=1pt] -- node [left] {} (1');
\draw (6') [line width=1pt] -- node [left] {} (3');
\draw (3') [line width=1pt] -- node [left] {} (1');
\draw (5') [line width=1pt] -- node [left] {} (1');
\draw (6') [line width=1pt] -- node [left] {} (1');
\end{tikzpicture}
\caption{Two isomorphic and asymmetric graphs $\mc{G}_1,\mc{G}_2$.}
\label{fig:sim_graph}
\end{figure}
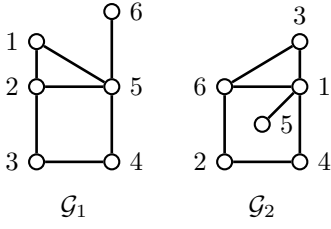

\begin{figure*}[t]
\centering
\begin{subfigure}[b]{0.32\textwidth}
\centering
\begin{tikzpicture}[scale=1.]
\node[place] (4) at (1,0.) [label=right:$4$] {};
\node[place] (3) at (0,0) [label=left:$3$] {};
\node[place] (1) at (0,1.6) [label=left:$1$] {};
\node[place] (2) at (0,1.) [label=left:$2$] {};
\node[place] (5) at (1.,1.) [label=right:$5$] {};
\node[place] (6) at (1.,2) [label=right:$6$] {};
\node[] (g1) at (0.5,-0.2) [label=below:$\mathcal{G}_1$] {};

\draw (2) [line width=1pt] -- node [left] {} (3);
\draw (3) [line width=1pt] -- node [left] {} (4);
\draw (4) [line width=1pt] -- node [left] {} (5);
\draw (2) [line width=1pt] -- node [left] {} (5);
\draw (2) [line width=1pt] -- node [left] {} (1);
\draw (1) [line width=1pt] -- node [left] {} (5);
\draw (6) [line width=1pt] -- node [left] {} (5);

\node[place] (4') at (1+2.5,0.+0.5) [label=right:$4$] {};
\node[place] (2') at (0+2.5,0+0.5) [label=below:$2$] {};
\node[place] (3') at (1+2.5,1.6+0.5) [label=above:$3$] {};
\node[place] (6') at (0+2.5,1.+0.5) [label=above:$6$] {};
\node[place] (1') at (1.+2.5,1.+0.5) [label=right:$1$] {};
\node[place] (5') at (0.5+2.5,0.5+0.5) [label=right:$5$] {};
\node[] (g1) at (0.5+2.5,-0.2+0.5) [label=below:$\mathcal{G}_2$] {};

\draw (2') [line width=1pt] -- node [left] {} (6');
\draw (2') [line width=1pt] -- node [left] {} (4');
\draw (4') [line width=1pt] -- node [left] {} (1');
\draw (6') [line width=1pt] -- node [left] {} (3');
\draw (3') [line width=1pt] -- node [left] {} (1');
\draw (5') [line width=1pt] -- node [left] {} (1');
\draw (6') [line width=1pt] -- node [left] {} (1');
\draw (5) [line width=1pt,dashed,red] -- node [left] {} (1');
\draw (3) [line width=1pt,dashed,red] -- node [left] {} (2');
\draw (4) [line width=1pt,dashed,red] -- node [left] {} (4');
\draw (1) [line width=1pt,dashed,red] -- node [left] {} (3');
\draw (6) [line width=1pt,dashed,red] -- node [left] {} (5');
\draw (2) [line width=1pt,dashed,red] -- node [left] {} (6');
\end{tikzpicture}
\caption{}
\label{fig:matching_result}
\end{subfigure}
\begin{subfigure}[b]{0.32\textwidth}
\centering
\includegraphics[height=4.3cm]{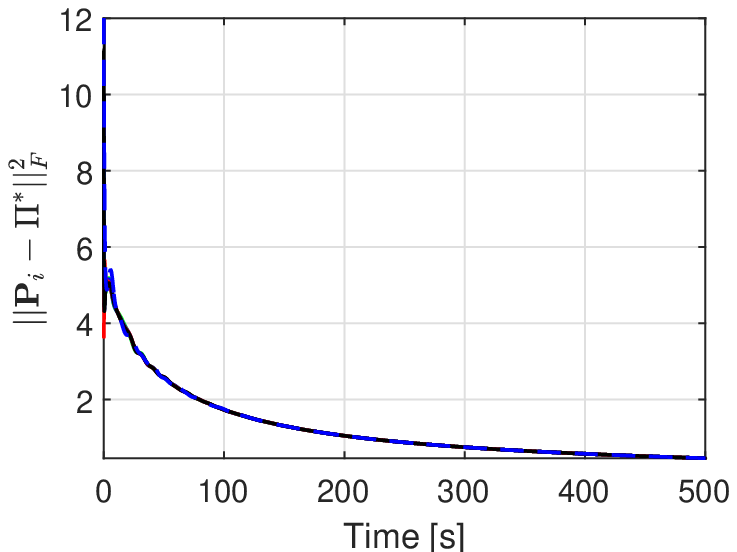}
\caption{}
\label{fig:deltaPi}
\end{subfigure}
\begin{subfigure}[b]{0.32\textwidth}
\centering
\includegraphics[height=4.3cm]{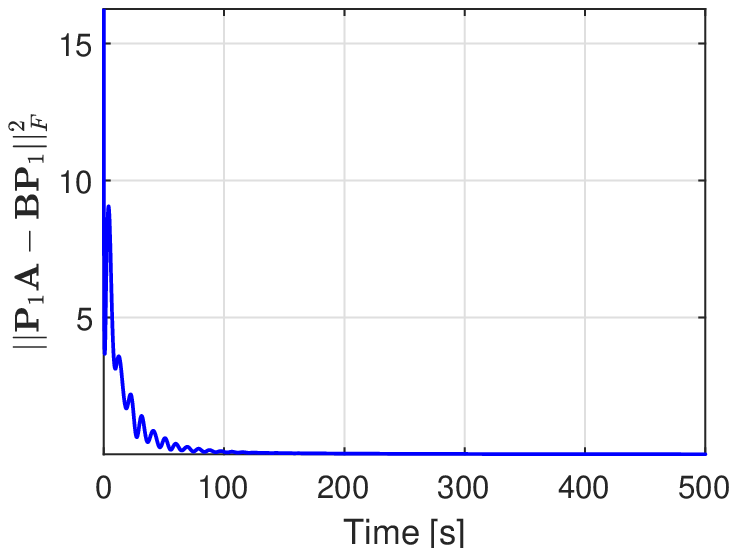}
\caption{}
\label{fig:distortion}
\end{subfigure}
\caption{Isomorphic graph matching between $\mc{G}_1$ and $\mc{G}_2$. (a) Matching results: vertex-to-vertex correspondences (dashed lines). (b) $||\m{P}_i-\m{\Pi}^*||_F^2,i=1,\ldots,6$ vs. time. (c) Adjacency disagreement $||\m{P}_1\m{A}-\m{B}\m{P}_1||_F^2$ vs. time.}
\label{fig:sim_graph_matching}
\end{figure*}

\section{Conclusion}\label{sec:conclusion}
In this work, we presented a distributed algorithm to compute the graph matching between two isomorphic and asymmetric graphs, over a multi-agent network. We first formulated the problem as a distributed optimization problem subject to equality constraints and a set constraint, and then proposed a continuous-time distributed algorithm to solve it. Given a small adjacency perturbation, we showed that the agents can compute the optimal permutation that matches two isomorphic graphs for all initial conditions and with exponential convergence. In addition, using the orthogonal projection onto the permutation set, the optimal permutation matrix can be obtained after a finite time. Simulation results were also provided.

There are several possible directions for future research. For example, it is desirable to investigate the distributed computation of graph matching of asymmetric graphs with large adjacency disagreement and graph matching of symmetric graphs. 

\appendix
\renewcommand{\thesection}{A\arabic{section}}
%
%
\subsection{Proof of Theorem \ref{thm:zero_measure_set}}\label{app:zero_measure_set}
Let $\m{A}\in \mb{R}^{n\times n}$ be a real adjacency matrix whose entries are nonnegative and has an eigenvector orthogonal to $\m{1}_n$. 
Let $\m{A}=\m{VAV}^\top$ be the eigenvalue decomposition of $\m{A}$, with $\m{V}=[\m{v}_1,\ldots,\m{v}_n]$ contains orthonormal eigenvectors corresponding to distinct eigenvalues $\lambda_i,i=1,\ldots,n$ of $\m{A}$. We will show that a perturbed version $\tilde{\m{A}}=\m{A} +\delta\m{A}$, with an arbitrary small perturbation, $\delta\m{A}\in \mb{R}^{n\times n}$, will have eigenvectors which are almost surely (a.s.) not orthogonal to $\m{1}_n$. Suppose that the perturbation $\delta\m{A}$ is symmetric and has zero diagonal entries $(\delta\m{A})_{jj}=0,\forall j=1,\ldots,n$. In addition, $(\delta\m{A})_{jk}\geq 0$ whenever $(\m{A})_{jk}=0,j\neq k$, and are real scalars otherwise, so that $\tilde{\m{A}}$ is a proper adjacency matrix for sufficiently small entries of $\delta\m{A}$.
The new eigenvalues and eigenvectors of $\tilde{\m{A}}$ can be expressed as $\tilde{\lambda}_i=\lambda_i+\delta\lambda_i$ and $\tilde{\m{v}}_i=\m{v}_i+\delta\m{v}_i$ with $\tilde{\m{v}}_i^\top\tilde{\m{v}}_i=1$, for some small $\delta\lambda_i\in \mb{R}$ and $\delta\m{v}_i\in \mb{R}^n$. Since $\{\m{v}_i\}_1^n$ is an orthonormal basis of $\mb{R}^n$, $\delta\m{v}_i=\sum_{k=1}^n\beta_{ik}\m{v}_k$ for a unique set of small scalars $\{\beta_{ik}\}_{k=1}^n$. Then, we have
\begin{align*}
(\m{A} +\delta\m{A})(\m{v}_i+\delta\m{v}_i)&=(\lambda_i+\delta\lambda_i)(\m{v}_i+\delta\m{v}_i)
\\
\Leftrightarrow \m{A}\delta\m{v}_i+\delta\m{A}\m{v}_i&=\lambda_i\delta\m{v}_i+\delta\lambda_i\m{v}_i
\\
\Leftrightarrow \m{A}\sum_{k=1}^n\beta_{ik}\m{v}_k+\delta\m{A}\m{v}_i&=\lambda_i\sum_{k=1}^n\beta_{ik}\m{v}_i+\delta\lambda_i\m{v}_i
\\
\Leftrightarrow \sum_{k=1}^n\beta_{ik}\lambda_k\m{v}_k+\delta\m{A}\m{v}_i&=\lambda_i\sum_{k=1}^n\beta_{ik}\m{v}_k+\delta\lambda_i\m{v}_i, \numberthis \label{eq:pertubed_eigenvector_value}
\end{align*}
where the second equality follows from $\m{Av}_i=\lambda_i\m{v}_i$ and by neglecting second-order terms. 
Left multiplying by $\m{v}_i^\top$ on both sides of the preceding relation gives
\begin{equation}
\delta\lambda_i = \m{v}_i^\top\delta\m{A}\m{v}_i\leq |\lambda_{max}(\delta\m{A})|. \label{eq:delta_lambda}
\end{equation}
Thus, by choosing $||\delta\m{A}||$ sufficiently small $\tilde{\m{A}}$ will still have simple spectrum.  
Left multiplying by $\m{v}_j^\top,j\neq i$ on both sides of \eqref{eq:pertubed_eigenvector_value} yields
\begin{align*}
\beta_{ij}\lambda_j+\m{v}_j^\top\delta\m{A}\m{v}_i&=\beta_{ij}\lambda_i\\
\Leftrightarrow \beta_{ij}&= f_{ij}(\delta\m{A})\triangleq\m{v}_j^\top\delta\m{A}\m{v}_i/(\lambda_i-\lambda_j),~\forall j\neq i. \numberthis \label{eq:beta_ij}
\end{align*}
From the unity condition $\tilde{\m{v}}_i^\top\tilde{\m{v}}_i=1\Leftrightarrow (\m{v}_i+\delta\m{v}_i)^\top(\m{v}_i+\delta\m{v}_i)=1\Leftrightarrow 1+2\m{v}_i^\top\delta\m{v}_i=1\Leftrightarrow\beta_{ii}=0$, where we use $||\delta\m{v}_i||^2\approx 0$. As a result, suppose that $\m{v}_i\perp \m{1}_n$, then $\tilde{\m{v}}_i\perp \m{1}_n$ when $\delta\m{v}_i=\sum_{k=1,k\neq i}^n\beta_{ik}\m{v}_k =\m{0}$ for sufficiently small $||\delta\m{A}||$. It follows that $\beta_{ik}=0\Leftrightarrow f_{ij}(\delta\m{A})=0$, for all $j=1,\ldots,n,j\neq i$, due to the mutually linear independence of $\{\m{v}_k\}$. Let $\delta \m{a}\in \mb{R}^{m_i}_+\times \mb{R}^{p_i}$ be a vector containing $m_i$ nonnegative upper-diagonal terms and the other $p_i=(n(n-1)/2-m_i)$ real upper-diagonal terms of $\delta\m{A}$, respectively.
Then, for a small open ball $\mc{B}_{\epsilon}(\m{0})$ the set 
$$\Omega_{\delta \m{a}}=\{\delta \m{a}\in \mc{B}_{\epsilon}(\m{0})\cap (\mb{R}^{m_i}_+\times \mb{R}^{p_i}):\delta\m{v}_i=\m{0}\}$$
is either a set of measure zero or the entire set $\mc{B}_{\epsilon}(\m{0})\cap(\mb{R}^{m_i}_+\times \mb{R}^{p_i})$ \cite{Caron2005}.
It follows from $\delta\m{A}\m{v}_i\perp\m{v}_j,\forall j=1,\ldots,n,j\neq i$ that $\delta\m{A}\m{v}_i=\gamma\m{v}_i$ for a scalar $\gamma$. Suppose that for all $\delta \m{a}\in \mc{B}_{\epsilon}(\m{0})\cap(\mb{R}^{m_i}_+\times \mb{R}^{p_i})$ there holds $\delta\m{A}\m{v}_i=\gamma\m{v}_i$ for a scalar $\gamma$ and nonzero $\m{v}_i=[{v}_{i1},\ldots,{v}_{in}]^\top$. Choose $\delta\m{A}=\left[\begin{smallmatrix}
0 &c &\m{0}\\
c &0 &\m{0}\\
\m{0} &\m{0} &\m{0}
\end{smallmatrix}\right]$ for a nonzero scalar $c$ and $\m{0}$ matrices of proper dimensions, then $\delta\m{A}\m{v}_i=\gamma\m{v}_i$ leads to $c v_{i2}=\gamma v_{i1}$ and $v_{ij}=0,\forall j=3,\ldots,n$. Similarly, select $(\delta\m{A})_{23}=(\delta\m{A})_{32}=c$ and the other entries are zeros, for a nonzero scalar $c$, then $\delta\m{A}\m{v}_i=\gamma\m{v}_i$ leads to $v_{i1}=0$. Consequently, $\m{v}_i\equiv \m{0}$, which is a contradiction, and hence $\Omega_{\delta \m{a}}$ is a set of measure zero.

Thus, $\tilde{\m{v}}_i$ is not orthogonal to $\m{1}_n$ for all $\m{A}$ outside a set of measure zero. This completes the proof. \eprf
\subsection{Proof of negative semi-definiteness of $\dot{V}$}\label{app:negative_dotV}
\begin{proof}
To proceed, we consider the following relations:
\begin{align*}
&\text{tr}\Big\{ \big[ -(\m{P}_i\m{a}_i-\m{y}_i)\m{a}_i^\top-\m{b}_i\boldsymbol{\lambda}_i^\top\\
&-\textstyle\sum_{j=1}^nw_{ij}(\Theta_i-\Theta_j)\big]^\top(\m{P}_i-\m{P}_i^*)\Big\}
\\
&\stackrel{\eqref{eq:KKT_Pi}}{=}\text{tr}\Big\{\Big[-\m{a}_i^\top\big((\m{P}_i-\m{P}_i^*)\m{a}_i-(\m{y}_i-\m{y}_i^*)\big)^\top\\
&-(\boldsymbol{\lambda}_i-\boldsymbol{\lambda}_i)\m{b}_i^\top 
\\
&-\textstyle\sum_{j=1}^nw_{ij}\big(\Theta_i-\Theta_j-(\Theta_i^*-\Theta_j^*)\big)\Big]^\top(\m{P}_i-\m{P}_i^*)\Big\}
\\
&=\text{tr}\Big\{-\Big[\big((\m{P}_i-\m{P}_i^*)\m{a}_i-(\m{y}_i-\m{y}_i^*)\big)^\top(\m{P}_i-\m{P}_i^*)\m{a}_i\\
&-(\boldsymbol{\lambda}_i-\boldsymbol{\lambda}_i)\m{b}_i^\top\\
&-\textstyle\sum_{j=1}^nw_{ij}\big(\Theta_i-\Theta_j-(\Theta_i^*-\Theta_j^*)\big)\Big]^\top(\m{P}_i-\m{P}_i^*)\Big\}. \numberthis \label{eq:relation1}
\end{align*}
Furthermore,
\begin{align*}
&\text{tr}\{\dot{\m{y}}_i^\top(\m{y}_i-\m{y}_i^*)\}=\text{tr}\Big\{\big[(\m{P}_i\m{a}_i-\m{y}_i)^\top-{(\Upsilon_i)_i^\text{C}}^\top\big]
\\
&\times(\m{y}_i-\m{y}_i^*)\Big\}\stackrel{\eqref{eq:KKT_yi}}{=}\text{tr}\Big\{\big[(\m{P}_i-\m{P}_i^*)\m{a}_i-(\m{y}_i-\m{y}_i^*)\big]^\top
\\
&\times(\m{y}_i-\m{y}_i^*) -\big({(\Upsilon_i)_i^\text{C}}^\top-{(\Upsilon_i^*)_i^\text{C}}^\top\big)(\m{y}_i-\m{y}_i^*)\Big\}.\numberthis \label{eq:relation2}
\end{align*}
Combining the first term in \eqref{eq:relation1} with the first term in \eqref{eq:relation2} gives
\begin{align*}
&\text{tr}\Big\{-\big((\m{P}_i-\m{P}_i^*)\m{a}_i-(\m{y}_i-\m{y}_i^*)\big)^\top(\m{P}_i-\m{P}_i^*)\m{a}_i\\
&+\big[(\m{P}_i-\m{P}_i^*)\m{a}_i-(\m{y}_i-\m{y}_i^*)\big]^\top(\m{y}_i-\m{y}_i^*)\Big\}\\
&=-||(\m{P}_i-\m{P}_i^*)\m{a}_i-(\m{y}_i-\m{y}_i^*)||^2. \numberthis \label{eq:relation3}
\end{align*}
Summing the two terms in the sum \eqref{eq:dotV} corresponding to the last term in \eqref{eq:Pi_dot} and the last term in \eqref{eq:z_i_dot} we have
\begin{align*}
&\text{tr}\big\{-(\m{b}_i^\top\m{P}_i-\m{z}_i^\top)^\top\m{b}_i^\top(\m{P}_i-\m{P}_i^*)\\
&\qquad+(\m{b}_i^\top\m{P}_i-\m{z}_i^\top)^\top(\m{z}_i-\m{z}_i^*)^\top\big\}\\
&=-\text{tr}\big\{(\m{b}_i^\top\m{P}_i-\m{z}_i^\top)^\top\big(\m{b}_i^\top(\m{P}_i-\m{P}_i^*)-(\m{z}_i-\m{z}_i^*)^\top\big) \big\}\\
&\stackrel{\eqref{eq:KKT_lamda}}{=}-||\m{b}_i^\top\m{P}_i-\m{z}_i^\top||^2.\numberthis \label{eq:relation4}
\end{align*}
The last term in \eqref{eq:Pi_dot} in the sum \eqref{eq:dotV} is given as
\begin{align*}
&\text{tr}\big\{ -\textstyle\sum_{i=1}^n\sum_{j=1}^nw_{ij}(\m{P}_i-\m{P}_j)^\top(\m{P}_i-\m{P}_i^*)\big\}\\
&\stackrel{\eqref{eq:KKT_Theta}}{=}-\textstyle\sum_{(i,j)\in \mc{E}_{\mc{H}}}w_{ij}(\m{P}_i-\m{P}_j)^\top(\m{P}_i-\m{P}_j)\big\}\\
&=-\textstyle\sum_{(i,j)\in \mc{E}_{\mc{H}}}w_{ij}||\m{P}_i-\m{P}_j||_F^2. \numberthis \label{eq:relation5}
\end{align*}
\begin{align*}
&\text{tr}\{\textstyle\sum_{i=1}^n\dot{\Theta_i}^\top(\Theta_i-\Theta_i^*)\}\\
&=\text{tr}\big\{\textstyle\sum_{i=1}^n\sum_{j=1}^nw_{ij}(\m{P}_i-\m{P}_j)^\top (\Theta_i-\Theta_i^*)\big\}\\
&\stackrel{\eqref{eq:KKT_Theta}}{=}\text{tr}\big\{\textstyle\sum_{(i,j)\in \mc{E}_{\mc{H}}}w_{ij}(\Theta_i-\Theta_j-(\Theta_i^*-\Theta_j^*))^\top(\m{P}_i-\m{P}_j)\big\}\\
&=\text{tr}\big\{\textstyle\sum_{i=1}^n\sum_{j=1}^nw_{ij}(\Theta_i-\Theta_j-(\Theta_i^*-\Theta_j^*))^\top(\m{P}_i-\m{P}_i^*)\big\}, \numberthis \label{eq:relation7}
\end{align*}
which cancels out the last term in \eqref{eq:relation1}. By \eqref{eq:K_i_dot} we have
\begin{align*}
&\text{tr}\big\{\textstyle\sum_{i=1}^n\dot{\m{K}}_i(\m{K}_i-\m{K}_i^*)\big\}\\
&=\text{tr}\big\{\textstyle\sum_{i=1}^n\sum_{j=1}^nw_{ij}(\Upsilon_i-\Upsilon_j)^\top(\m{K}_i-\m{K}_i^*) \big\}\\
&\stackrel{\eqref{eq:KKT_Ki}}{=}\text{tr}\Big\{\sum_{i=1}^n\sum_{j=1}^nw_{ij}(\Upsilon_i-\Upsilon_j-(\Upsilon_i^*-\Upsilon_j^*))^\top(\m{K}_i-\m{K}_i^*) \\
&=\text{tr}\big\{\textstyle\sum_{i=1}^n\sum_{j=1}^nw_{ij}(\Upsilon_i-\Upsilon_i^*)^\top(\m{K}_i-\m{K}_i^*) \\
&-\textstyle\sum_{i=1}^n\sum_{j=1}^nw_{ij}(\Upsilon_j-\Upsilon_j^*)^\top(\m{K}_i-\m{K}_i^*)\big\}. \numberthis \label{eq:relation8}
\end{align*}
Following \eqref{eq:Upsilon_i_dot} we have
\begin{align*}
&\text{tr}\{\textstyle\sum_{i=1}^n\dot{\Upsilon_i}^\top(\Upsilon_i-\Upsilon_i^*)\}\\
&=\text{tr}\Big\{\textstyle\sum_{i=1}^n\big([\m{Y}]_i-{[\m{Z}]^i}\big)^\top(\Upsilon_i-\Upsilon_i^*)\\
&-\textstyle\sum_{i=1}^n\sum_{j=1}^nw_{ij}\big[(\m{K}_i-\m{K}_j)-(\Upsilon_i-\Upsilon_j)\big]^\top(\Upsilon_i-\Upsilon_i^*)\Big\}
\\
&\stackrel{\eqref{eq:KKT_Upsilon}}{=}\text{tr}\Big\{\textstyle\sum_{i=1}^n\big([\m{Y}]_i-[\m{Z}]^i-[\m{Y}^*]_i+[\m{Z}^*]^i\big)^\top(\Upsilon_i-\Upsilon_i^*)\\
&-\textstyle\sum_{i=1}^n\sum_{j=1}^nw_{ij}\big(\m{K}_i-\m{K}_j-\m{K}_i^*+\m{K}_j^*\big)^\top(\Upsilon_i-\Upsilon_i^*)\Big\}\\
&-\textstyle\sum_{i=1}^n\sum_{j=1}^nw_{ij}(\Upsilon_i-\Upsilon_j)^\top(\Upsilon_i-\Upsilon_i^*)
\\
&=\text{tr}\Big\{\textstyle\sum_{i=1}^n\big([\m{Y}]_i-[\m{Z}]^i-[\m{Y}^*]_i+[\m{Z}^*]^i\big)^\top(\Upsilon_i-\Upsilon_i^*)\\
&-\textstyle\sum_{i=1}^n\sum_{j=1}^nw_{ij}\big(\m{K}_i-\m{K}_i^*\big)^\top(\Upsilon_i-\Upsilon_i^*)\\
&+\textstyle\sum_{i=1}^n\sum_{j=1}^nw_{ij}\big(\m{K}_j-\m{K}_j^*\big)^\top(\Upsilon_i-\Upsilon_i^*)\\
&-\textstyle\sum_{(i,j)\in \mc{E}_{\mc{H}}}w_{ij}||\Upsilon_i-\Upsilon_j||_F^2. \numberthis \label{eq:relation9}
\end{align*}
Combining \eqref{eq:relation8} and \eqref{eq:relation9} gives
\begin{align*}
&\eqref{eq:relation8}+\eqref{eq:relation9}=-\textstyle\sum_{(i,j)\in \mc{E}_{\mc{H}}}w_{ij}||\Upsilon_i-\Upsilon_j||_F^2\\
&+\text{tr}\Big\{\textstyle\sum_{i=1}^n\big([\m{Y}]_i-[\m{Y}^*]_i\big)^\top(\Upsilon_i-\Upsilon_i^*)\\
& -\textstyle\sum_{i=1}^n([\m{Z}]^i-[\m{Z}^*]^i)^\top(\Upsilon_i-\Upsilon_i^*)\Big\}. \numberthis \label{eq:relation10}
\end{align*}
Lastly, using the relation \eqref{eq:KKT_lamda} and \eqref{eq:KKT_zi}, we have
\begin{align*}
&\text{tr}\Big\{\textstyle\sum_{i=1}^n\boldsymbol{\dot{\lambda}}_i(\boldsymbol{\lambda}_i-\boldsymbol{\lambda}_i^*)^\top+\sum_{i=1}^n(\boldsymbol{\lambda}_i^\top+(\Upsilon_i)^\text{R}_i)^\top(\m{z}_i-\m{z}_i^*)^\top\Big\}
\\
&=\text{tr}\Big\{\textstyle\sum_{i=1}^n(\m{b}_i^\top(\m{P}_i-\m{P}_i^*)-(\m{z}_i-\m{z}_i^*)^\top)^\top(\boldsymbol{\lambda}_i-\boldsymbol{\lambda}_i^*)^\top
\\
&+\textstyle\sum_{i=1}^n(\boldsymbol{\lambda}_i^\top-{\boldsymbol{\lambda}_i^*}^\top+(\Upsilon_i)_i^\text{R}-(\Upsilon_i^*)_i^\text{R})^\top(\m{z}_i-\m{z}_i^*)^\top\Big\}
\\
&=\text{tr}\Big\{\textstyle\sum_{i=1}^n(\boldsymbol{\lambda}_i-\boldsymbol{\lambda}_i^*)\m{b}_i^\top(\m{P}_i-\m{P}_i^*)\\
&+\textstyle\sum_{i=1}^n((\Upsilon_i)_i^\text{R}-(\Upsilon_i^*)_i^\text{R})^\top(\m{z}_i-\m{z}_i^*)^\top\Big\},
\end{align*}
in which, the expression under the first sum is cancelled out by the second term in \eqref{eq:relation1}, and the trace of the last sum is compensated by the trace of the last sum in \eqref{eq:relation10}.

By combining all the preceding relations, we obtain \eqref{eq:dotV_simpler}. 
\end{proof}
\section*{Acknowledgments}
The work of Q. V. Tran and H.-S. Ahn was supported by the National Research Foundation of Korea (NRF) grant funded by the Korea government (MSIT) (2019R1A4A1029003).

B. D. O. Anderson is supported by the Australian Research Council
(ARC) under grant DP-160104500 and DP-190100887, also by Data61-
CSIRO.


\nocite{}
\bibliographystyle{IEEEtran}
\bibliography{quoc2018,quoc2019}

\end{document}